\documentclass{amsart}
\usepackage{latexsym,amsxtra,amscd,ifthen}
\usepackage{amsfonts}
\usepackage{verbatim}
\usepackage{amsmath}
\usepackage{amsthm}
\usepackage{amssymb}

\input xy \xyoption{matrix} \xyoption{arrow}\xyoption{frame}
\def\dropdown#1{\save+<0ex,-6ex>\drop{#1}\restore}
\def\dropdownmore#1{\save+<0ex,-10ex>\drop{#1}\restore}
\def\dropdownright#1{\save+<8ex,-6ex>\drop{#1}\restore}
\def\dropdownrightmore#1{\save+<8ex,-10ex>\drop{#1}\restore}
\newcommand{\edge}{\ar@{-}}
\newcommand{\place}{*{}+0}
\newcommand{\plb}{*{\bullet}+0}
\newcommand{\plc}{*{\circ}+0}

\numberwithin{equation}{section}

\theoremstyle{plain}
\newtheorem{theorem}{Theorem}[section]
\newtheorem{lemma}[theorem]{Lemma}
\newtheorem{proposition}[theorem]{Proposition}

\newtheorem{conjecture}[theorem]{Conjecture}

\theoremstyle{definition}
\newtheorem{definition}[theorem]{Definition}
\newtheorem{example}[theorem]{Example}

\newtheorem{remark}[theorem]{Remark}
\newtheorem*{remark*}{Remark}

\newtheorem{obs}[theorem]{Observation}

\newtheorem{assumps}[theorem]{Assumptions}
\newtheorem{subsec}[theorem]{} 

\newcommand{\calC}{{\mathcal C}}
\newcommand{\calE}{{\mathcal E}}
\newcommand{\calO}{{\mathcal O}}
\newcommand{\calS}{{\mathcal S}}

\renewcommand{\AA}{{\mathbb A}}

\newcommand{\Itil}{\widetilde{I}}
\newcommand{\Atil}{\widetilde{A}}
\newcommand{\Jtil}{\widetilde{J}}
\newcommand{\calEtil}{\widetilde{\calE}}

\newcommand{\Sbar}{\overline{S}}
\newcommand{\Xbar}{\overline{X}}
\newcommand{\Ybar}{\overline{Y}}

\newcommand{\Oq}{\calO_q}

\newcommand{\OqMn}{\Oq(M_n(k))}
\newcommand{\OqGLth}{\Oq(GL_3(k))}

\newcommand{\OqSLth}{\Oq(SL_3(k))}
\newcommand{\OqSLn}{\Oq(SL_n(k))}
\newcommand{\kx}{k^\times}
\newcommand\kbar{\overline{k}}

\newcommand\hightop{\overline{\,|\,}}

\DeclareMathOperator\ident{\operatorname{id}}

\newcommand{\gnoc}{\mathrel{{\lower.2ex\hbox{$\backsim$}}\llap{\raise.45ex\hbox{=}}}}

\newcommand{\Hspec}{H\text{-}\operatorname{spec}}

\DeclareMathOperator\spec{\operatorname{spec}}
\DeclareMathOperator\prim{\operatorname{prim}}
\DeclareMathOperator\Fract{\operatorname{Fract}}

\DeclareMathOperator\hgt{\operatorname{ht}}
\DeclareMathOperator\CL{\operatorname{CL}}
\DeclareMathOperator\rank{\operatorname{rank}}

\newcommand{\ccc}{\plc&\plc&\plc}
\newcommand{\bcc}{\plb&\plc&\plc}

\newcommand{\ccb}{\plc&\plc&\plb}
\newcommand{\bbc}{\plb&\plb&\plc}
\newcommand{\bcb}{\plb&\plc&\plb}
\newcommand{\cbb}{\plc&\plb&\plb}

\newcommand{\hrz}{\place \edge[r] &\place}

\newcommand{\hrzvrt}{\place \edge[r] \edge[d] &\place \edge[d]}

\newcommand{\dropvarup}[2]{\save+<0ex,#1ex>\drop{#2}\restore}
\newcommand{\dropup}[1]{\save+<0ex,4ex>\drop{#1}\restore}

\newcommand{\dropgen}[3]{\save+<#1ex,#2ex>\drop{#3}\restore}

\begin{document}

\title[Zariski topologies on stratified spectra of quantum algebras] {Zariski topologies on stratified spectra of
quantum algebras}

\author[K.A. Brown]{K.A. Brown} \address{Department of Mathematics\\ University of Glasgow\\ Glasgow G12 8QW\\
Scotland.} \email{Ken.Brown@glasgow.ac.uk}

\author[K.R. Goodearl]{K.R. Goodearl} \address{Department of Mathematics\\ University of California\\ Santa Barbara,
CA 93106 \\ USA.} \email{goodearl@math.ucsb.edu}

\thanks{The research of the second-named author was supported by grant DMS-0800948 from the National Science Foundation (USA). Parts of this project were completed while the authors were in residence at the Mathematical Sciences Research Institute during its 2013 program on Noncommutative Algebraic Geometry and Representation 
Theory.}

\begin{abstract}
A framework is developed to describe the Zariski topologies on the prime and primitive spectra of a quantum algebra $A$ in terms of the (known) topologies on strata of these spaces and maps between the collections of closed sets of different strata. A conjecture is formulated, under which the desired maps would arise from homomorphisms between certain central subalgebras of localized factor algebras of $A$. When the conjecture holds, $\spec A$ and $\prim A$ are then determined, as topological spaces, by a finite collection of (classical) affine algebraic varieties and morphisms between them. The conjecture is verified for $\Oq(GL_2(k))$, $\OqSLth$, and $\Oq(M_2(k))$ when $q$ is a non-root of unity and the base field $k$ is algebraically closed.
\end{abstract}

\maketitle


\section{Introduction} \label{intro}

For many quantum algebras $A$, by which we mean quantized coordinate rings, quantized Weyl algebras, and related algebras, good piecewise pictures of the prime and primitive spectra are known. More precisely, in generic cases there are finite stratifications of these spectra, based on a rational action of an algebraic torus, such that each stratum is homeomorphic to the prime or primitive spectrum of a commutative Laurent polynomial ring. What is lacking is an understanding of how these strata are combined topologically, i.e., of the Zariski topologies on the full spaces $\spec A$ and $\prim A$. We develop a framework for the needed additional data, in terms of maps between the collections of closed sets of different strata, together with a conjecture stating how these maps should arise from homomorphisms between certain central subalgebras of localizations of factor algebras of $A$.

In the stratification picture just mentioned (see Theorem \ref{strat.thm} for details), each stratum is ``classical" in that it is homeomorphic to either a classical affine algebraic variety or the scheme of irreducible closed subvarieties of an affine variety. One would like $\spec A$ and $\prim A$ themselves to be fully describable in terms of classical data. This is a key aspect of our main goal: to formulate a conjectural picture which describes the topological spaces $\spec A$ and $\prim A$ in terms of completely classical data, namely a finite collection of affine varieties together with suitable morphisms between them. We verify this picture in three basic cases -- the generic quantized coordinate rings of the groups $GL_2(k)$ and $SL_3(k)$, and of the matrix variety $M_2(k)$.

Our analysis  of the described picture brings with it new structural information about the algebras $\OqSLth$ and $\Oq(M_2(k))$. All prime factor rings of these algebras are Auslander-Gorenstein and Cohen-Macaulay with respect to GK-dimension, and all but one of the factor rings modulo prime ideals invariant under the natural acting tori are noncommutative unique factorization domains in the sense of Chatters \cite{Cha}. The exceptional case gives an example of a noetherian domain (and maximal order) with infinitely many height $1$ prime ideals, all but exactly four of which are principal. This is a previously unobserved phenomenon, which does not occur in the commutative case \cite{Bouv}.
\medskip

Throughout, we work over an algebraically closed base field $k$, of arbitrary characteristic.


\section{Stratified topological data}
\label{strat.topo}

Determining the global Zariski topology on the prime or primitive spectrum of a quantum algebra, given knowledge of the subspace topologies on all strata, requires some relations between the topologies of different strata. We give such relations in terms of maps between collections of closed sets. An abstract framework for this data is developed in the present section.

We denote the closure of a set $S$ in a topological space by $\Sbar$.

\begin{definition} \label{finstrat} 
A \emph{finite stratification} of a topological space $T$ is a finite partition $T
= \bigsqcup\, \{ S \in \calS \}$ such that 
\begin{enumerate} 
\item Each set in $\calS$ is a nonempty locally closed
subset of $T$. 
\item The closure of each set in $\calS$ is a union of sets from $\calS$. 
\end{enumerate} 
In this
setting, we define a relation $\le$ on $\calS$ by the rule 
\begin{equation}  \label{S2S1bar} 
S \le S' \quad\iff\quad
S' \subseteq \Sbar\,, 
\end{equation} 
and we observe as follows that $\le$ is a partial order. Reflexivity and
transitivity are clear. If $S_1, S_2 \in \calS$ satisfy $S_1 \le S_2$ and $S_2 \le S_1$, then $\Sbar_1 = \Sbar_2$.
Inside this closed set, $S_1$ and $S_2$ are both dense and open (by condition (2)), so $S_1 \cap S_2 \ne \varnothing$,
and consequently $S_1 = S_2$.

In view of the above observation, it is convenient to present finite stratifications as partitions indexed by finite
posets. Consequently, we rewrite the definition in the following terms.

A \emph{finite stratification} of a topological space $T$ is a partition $T = \bigsqcup_{i\in \Pi} S_i$ such that
\begin{enumerate} 
\item[(3)] $\Pi$ is a finite poset. 
\item[(4)] Each $S_i$ (for $i\in \Pi$) is a nonempty locally
closed subset of $T$. 
\item[(5)] For each $i \in \Pi$, the closure of $S_i$ in $T$ is given by $\Sbar_i =
\bigsqcup_{\, j\in\Pi,\; j\ge i}
    S_j$.
\end{enumerate} 
Observe that the ordering on $\Pi$ matches that of \eqref{S2S1bar}. Namely, for $i,j \in \Pi$, we
have 
\begin{equation}  \label{SjSibar} 
i \le j \quad\iff\quad S_j \subseteq \Sbar_i\,. 
\end{equation}
\end{definition}

\begin{definition} \label{CLetc} 
We shall write $\CL(T)$ to denote the collection of all closed subsets of a
topological space $T$.

Suppose that $T = \bigsqcup_{\, i\in \Pi} S_i$ is a finite stratification of $T$. For $i<j$ in $\Pi$, define a map
$\phi_{ij} : \CL(S_i) \rightarrow \CL(S_j)$ by the rule 
$$\phi_{ij}(Y) = \Ybar \cap S_j\,.$$ 
(These maps can be
defined for any pair of elements $i,j \in \Pi$, but the cases in which $i \not< j$ will not be needed.) The family
$(\phi_{ij})_{i,j\in\Pi,\, i< j}$ will be referred to as the \emph{associated family of maps} for the given
stratification.
 \end{definition}

\begin{lemma} \label{topostrat} 
Let $T$ be a topological space with a finite stratification $T = \bigsqcup_{\, i\in \Pi}
S_i$, and let $(\phi_{ij})_{i,j\in\Pi,\, i< j}$ be the associated family of  maps.

{\rm(a)} Each $\phi_{ij}$ maps $\varnothing \mapsto \varnothing$ and $S_i \mapsto S_j$.

{\rm(b)} Each $\phi_{ij}$ preserves finite unions.

{\rm(c)} A subset $X\subseteq T$ is closed in $T$ if and only if 
\begin{enumerate} 
\item $X\cap S_i \in \CL(S_i)$ for
all $i \in \Pi$; and 
\item $\phi_{ij}(X\cap S_i) \subseteq X\cap S_j$ for all $i< j$ in $\Pi$.
\end{enumerate}
\end{lemma}

\begin{proof} Statements (a) and (b) are clear.

(c) If $X$ is a closed subset of $T$, then (1) is obvious. As for (2): Given $i < j$ in $\Pi$, we see that
$$\phi_{ij}(X\cap S_i) = \overline{X\cap S_i} \cap S_j \subseteq \Xbar \cap \Sbar_i \cap S_j = X\cap S_j\,,$$ 
taking
account of \eqref{SjSibar}.

Conversely, let $X$ be a subset of $T$ for which (1) and (2) hold. Write $X= \bigsqcup_{\, i\in\Pi} X_i$, where $X_i :=
X\cap S_i$. By our assumptions, $X_i \in \CL(S_i)$ for all $i$ and $\phi_{ij}(X_i) \subseteq X_j$ for all $i< j$. Set
$Y := \bigcup_{\, i\in\Pi} \Xbar_i$, which is closed in $T$ because $\Pi$ is finite. Obviously, $X\subseteq Y$ and $Y =
\bigcup_{\, i,j\in\Pi} \Xbar_i \cap S_j$. Consider $i,j\in \Pi$ such that $\Xbar_i \cap S_j \ne \varnothing$. If $i=j$,
then $\Xbar_i \cap S_j = X_i \subseteq X$. Now assume that $i \ne j$. Then $\Sbar_i \cap S_j \ne \varnothing$, whence
$S_j \subseteq \Sbar_i$ and $i< j$ (by condition (5) of Definition \ref{finstrat}). Consequently, $\Xbar_i \cap S_j
= \phi_{ij}(X_i) \subseteq X_j \subseteq X$. We have now shown that $\Xbar_i \cap S_j \subseteq X$ for all $i,j \in
\Pi$, and thus $Y=X$. This shows that $X$ is closed in $T$, and completes the proof. 
\end{proof}

\begin{remark}  \label{topodata} 
We mention that data of the above kind can be used to construct topologies, as
follows. Suppose that $\Pi$ is a finite poset, $(S_i)_{i\in\Pi}$ is a family of topological spaces indexed by $\Pi$,
and maps $\phi_{ij}: \CL(S_i) \rightarrow \CL(S_j)$ are given for all $i< j$ in $\Pi$. Arrange for the spaces $S_i$
(or suitable copies of them) to be pairwise disjoint, and set $T := \bigsqcup_{\, i\in\Pi} S_i$. Assume that conditions
(a) and (b) of Lemma \ref{topostrat} hold, and let $\calC$ be the collection of those subsets $X$ of $T$ satisfying
conditions (c)(1), (c)(2) of the lemma. Then $\calC$ is the collection of closed sets for a topology on $T$, and the
partition $T = \bigsqcup_{\, i\in\Pi} S_i$ is a finite stratification. We leave the easy proof to the reader.
\end{remark}


\section{$H$-strata} \label{Hstrata} 

In this section, we review the toric stratifications of the spectra of quantum algebras and develop maps that, conjecturally, provide the data needed to invoke the framework of Section \ref{strat.topo}.

\begin{assumps}  \label{AHetc} 
In general, we will work with algebras $A$ and tori $H$ satisfying the following
conditions: 
\begin{enumerate} 
\item $A$ is a noetherian $k$-algebra, satisfying the noncommutative Nullstellensatz
over $k$. 
\item $H$ is a $k$-torus, acting rationally on $A$ by $k$-algebra automorphisms. \item $A$ has only finitely
many $H$-prime ideals. 
\end{enumerate}
See, e.g., \cite[\S9.1.4]{McRob} for the definition of the noncommutative Nullstellensatz
over $k$, and \cite[\S II.2]{BG} for a discussion of rational actions. 

It is standard to denote the set of all $H$-prime ideals (= $H$-stable prime ideals) of $A$ by $\Hspec A$. By
assumption (3), this set is finite, and we view it as a poset with respect to $\subseteq$. Thus, we will often take
$\Pi = \Hspec A$.

Recall that for $J \in \Hspec A$, the \emph{$J$-stratum} of $\spec A$ is the set 
$$\spec_J A := \{ P \in \spec A \mid
\bigcap_{h\in H} h.P = J \},$$ 
and the corresponding $J$-stratum in $\prim A$ is 
$$\prim_J A := \bigl( \spec_J A
\bigr) \cap \prim A.$$ 
These sets give finite stratifications of $\spec A$ and $\prim A$ (see Observation
\ref{phiJK}).

We shall express the closed subsets of $\spec A$ and $\prim A$ in the forms 
$$V(I) := \{ P\in \spec A \mid P \supseteq
I \}  \qquad\text{and}\qquad V_p(I) := \{ P\in \prim A \mid P \supseteq I \},$$ 
for ideals $I$ of $A$.

The rational action of $H$ on $A$ makes $A$ a graded algebra over the character group $X(H)$ (cf.~\cite[Lemma
II.2.11]{BG}). The nonzero homogeneous elements for this grading are precisely the $H$-eigenvectors. It will be
convenient to express many statements in terms of homogeneous elements rather than $H$-eigenvectors, in $A$ as well as
in factors of $A$ modulo $H$-primes and localizations thereof. This also allows us to refer to homogeneous components
of elements. (Since the mentioned $X(H)$-gradings are the only gradings used in this paper, we may use the term
``homogeneous'' without ambiguity.) Now $X(H)$ is a free abelian group of finite rank, so it can be made into a
totally ordered group in various ways. Fix such a totally ordered abelian group structure on $X(H)$. This allows us to
refer to \emph{leading terms} and \emph{lowest degree terms} of nonhomogeneous elements when needed. 
\end{assumps}

For reference, we quote the parts of the Stratification and Dixmier-Moeglin Equivalence Theorems (\cite[Theorems
II.2.13, II.8.4, Proposition II.8.3]{BG}) relevant to our present work.

\begin{theorem}  \label{strat.thm} 
Impose Assumptions {\rm\ref{AHetc}}, and let $J \in \Hspec A$.

{\rm(a)} The set $\calE_J$ of all regular homogeneous elements in $A/J$ is a denominator set, and the localization
$A_J := (A/J)[\calE_J^{-1}]$ is an $H$-simple ring {\rm(}with respect to the induced $H$-action{\rm)}.

{\rm(b)} $\spec_J A \approx \spec A_J \approx \spec Z(A_J)$ via localization, contraction, and extension.

{\rm(c)} $Z(A_J)$ is a Laurent polynomial ring over $k$ in at most $\rank H$ indeterminates.

{\rm(d)} $\prim_J A$ equals the set of maximal elements of $\spec_J A$, and the maps in {\rm(b)} restrict to a
homeomorphism $\prim_J A \approx \max Z(A_J)$. 
\end{theorem}

When working with specific algebras such as $\OqSLn$ or $\OqMn$, it may be convenient to shrink the denominator sets
$\calE_J$. This can be done without loss of the above properties in the following circumstances.

\begin{lemma}  \label{shrinkEJ} 
Impose Assumptions {\rm\ref{AHetc}}, and let $J \in \Hspec A$. Suppose that $\calE
\subseteq \calE_J$ is a denominator set such that all nonzero $H$-primes of $A/J$ have nonempty intersection with
$\calE$. Then:

{\rm(a)} The localization $A_\calE := (A/J)[\calE^{-1}]$ is $H$-simple.

{\rm(b)} $\spec_J A \approx \spec A_\calE \approx \spec Z(A_\calE)$ and $\prim_J A \approx \max Z(A_\calE)$ via
localization, contraction, and extension.

{\rm(c)} $Z(A_J) = Z(A_\calE)$. 
\end{lemma}

\begin{proof} Similar observations have been made in a number of instances, such as \cite[\S3.2]{primOqSL3}. We repeat
the arguments for the reader's convenience.

(a) Any $H$-prime of $A_\calE$ contracts to an $H$-prime of $A/J$ disjoint from $\calE$, and is thus zero by virtue of
our hypothesis on $\calE$. Consequently, $A_\calE$ has no nonzero $H$-primes, and therefore it is $H$-simple.

(b) Note that all nonzero $H$-ideals of $A/J$ have nonempty intersection with $\calE$, because $A_\calE$ is
$H$-simple.

The $J$-stratum $\spec_J A$ may be rewritten in the form 
$$\spec_J A = \{ P \in \spec A \mid P \supseteq J
\text{\;and\;} P/J \text{\;contains no nonzero\;} H\text{-ideals of\;} A/J \},$$ 
from which we see that 
$$\spec_J A =
\{ P \in \spec A \mid P \supseteq J \text{\;and\;} (P/J)\cap \calE = \varnothing \}.$$ 
Consequently, localization
provides a homeomorphism $\spec_J A \approx \spec A_\calE$. The homeomorphism $\spec A_\calE \approx \spec Z(A_\calE)$
follows from \cite[Corollary II.3.9]{BG} because $A_\calE$ is $H$-simple. Finally, because $\prim_J A$ is the
collection of maximal elements in $\spec_J A$, the composite homeomorphism $\spec_J A \rightarrow \spec Z(A_\calE)$
restricts to a homeomorphism $\prim_J A \rightarrow \max Z(A_\calE)$.

(c) Since $Z(A_\calE)$ is central in $\Fract A/J$, we must have $Z(A_\calE) \subseteq Z(A_J)$. Conversely, consider an
element $c \in Z(A_J)$. As is easily checked, the homogeneous components of $c$ are all central (e.g., \cite[Exercise
II.3.B]{BG}), and so to prove that $c \in Z(A_\calE)$, there is no loss of generality in assuming that $c$ itself is
homogeneous. Set $I := \{ a \in A_\calE \mid ac \in A_\calE \}$, and observe that $I$ is a nonzero $H$-stable ideal of
$A_\calE$ (it is nonzero because $A_J$ is a localization of $A_\calE$). Since $A_\calE$ is $H$-simple, we have $I =
A_\calE$, whence $c \in A_\calE$ and thus $c \in Z(A_\calE)$. 
\end{proof}

\begin{obs}  \label{phiJK} 
Under Assumptions \ref{AHetc}, we have partitions 
\begin{equation}
\label{specprimpartitions} 
\spec A = \bigsqcup_{J\in\Pi} \spec_J A  \qquad\text{and}\qquad  \prim A =
\bigsqcup_{J\in\Pi} \prim_J A \,, 
\end{equation} 
where $\Pi := \Hspec A$. These partitions are finite stratifications,
because 
\begin{align*} 
\spec_J A &= V(J) \setminus \biggl(\, \bigsqcup_{K\in\Pi,\,K \supsetneq J} V(K) \biggr)  \\
\overline{\spec_J A} &= V(J) = \bigsqcup_{K\in\Pi,\,K\supseteq J} \spec_K A 
\end{align*} 
for $J \in \Pi$, and
similarly for $\prim_J A$ and its closure. The last step requires the fact that $\overline{\prim_J A} = V_p(J)$. We
shall later need a slight generalization: 
\begin{equation}  \label{VpPprimclosure} 
\overline{V_p(P) \cap \prim_J A} =
V_p(P) \text{\;for all\;} P \in \spec_J A. 
\end{equation} 
This follows from the assumption that $A$ is a Jacobson
ring, as in \cite[Proposition 1.3(a)]{BGtams}; we include the short argument. Any primitive ideal of $A$ that contains
$P$ also contains $J$, so it belongs to $\prim_L A$ for some $H$-prime $L \supseteq J$. Hence, $$P = \bigcap\, \{ Q
\in \prim A \mid Q \supseteq P \} = \bigcap_{L \in \Pi,\, L\supseteq J} V_p(P)\cap \prim_L A\,.$$ Since $\Hspec A$ is
finite and $\bigcap \bigl( V_p(P) \cap \prim_L A \bigr) \supseteq L \supsetneq J$ for all $H$-primes $L$ that properly
contain $J$, we conclude that 
\begin{equation}  
\label{PcapVpPprimJ} P = \bigcap\, \bigl( V_p(P) \cap \prim_J A \bigr)
\text{\;for all\;} P \in \spec_J A. 
\end{equation} 
This implies \eqref{VpPprimclosure}.

We shall use the following notation for the maps described in Definition \ref{CLetc} relative to the above
stratifications: 
\begin{equation}  \label{phiJKdef} 
\begin{aligned} 
\phi^s_{JK} : \CL(\spec_J A) &\longrightarrow
\CL(\spec_K A), &\phi^s_{JK}(Y) &= \Ybar \cap \spec_K A  \\ \phi^p_{JK} : \CL(\prim_J A) &\longrightarrow \CL(\prim_K
A), &\phi^p_{JK}(Y) &= \Ybar \cap \prim_K A 
\end{aligned} 
\end{equation} 
for $J \subset K$ in $\Pi$.

In view of Lemma \ref{topostrat}, the Zariski topologies on $\spec A$ and $\prim A$ are determined by the topologies
on the strata $\spec_J A$ and $\prim_J A$ together with the maps $\phi^\bullet_{JK}$. Since the spaces $\spec_J A$ and
$\prim_J A$ are given (and computible) by Theorem \ref{strat.thm}, what remains is to determine the maps
$\phi^\bullet_{JK}$. 
\end{obs}

\begin{example}  \label{eg.oqk2} 
Let $A = \Oq(k^2)$ with $q$ not a root of unity, standard generators $x$, $y$, and
the standard action of $H = (\kx)^2$. (E.g., see \cite[Examples II.1.6(a), II.2.3(a), II.8.1]{BG}.) Consider the
$H$-primes $J := \langle x\rangle$ and $K := \langle x,y\rangle$, and recall that 
\begin{align*} 
\prim_J A &= \{
\langle x, y-\beta\rangle \mid \beta \in \kx\}  &\spec_J A &= \{ J \} \sqcup \prim_J A  \\ \prim_K A &= \spec_K A = \{
K \}\,. \end{align*} 
The maps $\phi^\bullet_{JK}$ can be described as follows: \begin{align*} \phi^s_{JK}(Y) &=
\begin{cases} \varnothing &(Y \text{\;finite,\;} J \notin Y)\\ \{  K \} &(Y \text{\;infinite or\;} J \in Y)
\end{cases}  &&(Y \in \CL(\spec_J A)  \\ \phi^p_{JK}(Y) &= \begin{cases} \varnothing &(Y \text{\;finite})\\ \{  K \}
&(Y \text{\;infinite}) \end{cases}  &&(Y \in \CL(\prim_J A) . 
\end{align*}

Observe that the two ``natural'' possibilities for maps between collections of closed sets are ruled out by the fact
that for primitive ideal strata, $\phi^p_{JK}$ maps all singletons to the empty set. Namely, there is no continuous
map $f : \prim_K A \rightarrow \prim_J A$ such that $\phi^p_{JK}(Y) = f^{-1}(Y)$ for $Y \in \CL(\prim_J A)$, and there
is no map $g : \prim_J A \rightarrow \prim_K A$ such that $\phi^p_{JK}(Y) = \overline{g(Y)}$ for $Y \in \CL(\prim_J
A)$. Nor can $\phi^s_{JK} : \spec_J A \rightarrow \spec_K A$ be described in either of these ways.

On the other hand, $\phi^p_{JK}$ can easily be obtained from a combination of two such maps. For instance, we can
define continuous maps $f: \prim_K A \rightarrow \AA^1_k$ and $g: \prim_J A \rightarrow \AA^1_k$ by the rules
$$f(\langle x,y\rangle) = 0  \qquad\text{and}\qquad g(\langle x,\, y-\beta\rangle) = \beta,$$ 
with the help of which
$\phi^p_{JK}$ can be expressed in the form 
$$\phi^p_{JK}(Y) = f^{-1} \bigl( \overline{g(Y)} \bigr)$$ 
for $Y \in
\CL(\prim_J A)$.

It will be convenient to introduce the following notation for maps of this type. 
\end{example}

\begin{definition}  \label{ftopg} 
Suppose that $f: S' \rightarrow W$ and $g: S \rightarrow W$ are continuous maps
between topological spaces. We define a map 
$$f\hightop g : \CL(S) \longrightarrow \CL(S')$$ 
according to the rule
$$(f\hightop g)(Y) = f^{-1} \bigl( \overline{g(Y)} \bigr).$$ 
(The notation $f\hightop g$ is meant to abbreviate
$f^{-1} \circ \overline{(-)} \circ g$.) 
\end{definition}

\begin{remark}  \label{topwanted} 
Under Assumptions \ref{AHetc}, we would like good descriptions of the  maps
$\phi^\bullet_{JK}$ (for $J \subset K$ in $\Hspec A$) in the form $f\hightop g$. There is always a trivial way to do
this. For instance, if we let $f: \spec_K A \rightarrow \spec A$ and $g: \spec_J A \rightarrow \spec A$ be the
inclusion maps, then $\phi^s_{JK} = f\hightop g$ by definition of $\phi^s_{JK}$. However, this is no help towards our
goal of describing the topological space $\spec A$.

By the Stratification Theorem \ref{strat.thm}, each $\prim_J A$ is the topological space underlying an affine variety
$\max Z(A_J)$ over $k$, and $\spec_J A$ is the space underlying the corresponding scheme $\spec Z(A_J)$. In the first
case, it is natural to ask for $\phi^p_{JK} = f\hightop g$ where $f$ and $g$ are morphisms of varieties, and in the
second case, to ask for $\phi^s_{JK} = f\hightop g$ where $f$ and $g$ are morphisms of schemes. In both cases, $f$ and
$g$ would be comorphisms of $k$-algebra maps $R \rightarrow Z(A_K)$ and $R\rightarrow Z(A_J)$, for some affine
commutative $k$-algebra $R$. Given the forms of $A_J$ and $A_K$, it is natural to conjecture that an appropriate $R$
would be the center of some localization of $A/J$, specifically, the localization of $A/J$ with respect to the set
$\calE_{JK}$ of those homogeneous elements of $A/J$ which are regular modulo $K/J$. However, such a localization does
not always exist, even in case $H$ is trivial and $A$ has only finitely many prime ideals. On the other hand, if
$(A/J)[\calE_{JK}^{-1}]$ did exist, its center could be described in the form 
$$Z((A/J)[\calE_{JK}^{-1}]) = \{ z \in
Z(A_J) \mid zc \in A/J \text{\;\;for some\;} c \in \calE_{JK} \},$$ 
which does not require the existence of
$(A/J)[\calE_{JK}]^{-1}]$. Thus, we propose to work with algebras of the latter type. 
\end{remark}

\begin{definition}  \label{ZJK} 
Impose Assumptions \ref{AHetc}. For $J \subset K$ in $\Hspec A$, set 
\begin{align}
\calE_{JK} &:= \{ \text{homogeneous elements\;} c \in A/J \mid c \text{\;is regular modulo\;} K/J \}    \label{defZJK}
\\ Z_{JK} &:= \{ z \in Z(A_J) \mid zc \in A/J \text{\;for\;} \text{some\;} c \in \calE_{JK} \}.  \label{defEJK}
\end{align} 
It is easily checked that $Z_{JK}$ is a $k$-subalgebra of $Z(A_J)$. For, given any $z_1,z_2 \in Z(A_J)$,
there exist $c_1,c_2 \in \calE_{JK}$ such that $z_ic_i \in A/J$ for $i=1,2$, whence $c_1c_2 \in \calE_{JK}$ and
\begin{equation}  \label{ZJKsubalg} 
\begin{aligned} 
(z_1z_2)(c_1c_2) &= z_1c_1z_2c_2 \in A/J  \\ (z_1\pm z_2)(c_1c_2)
&= z_1c_1c_2 \pm c_1z_2c_2 \in A/J. 
\end{aligned} 
\end{equation}

Note also that $Z_{JK} \supseteq Z(A/J)$.

In general, it appears that we must allow the possibility that $Z_{JK}$ might not be affine, although that will be the
case in all the examples we analyze. This is not a problem, however, since we are only concerned with $\max Z_{JK}$
and $\spec Z_{JK}$ as topological spaces. 
\end{definition}

In examples, $Z_{JK}$ can often be computed as the center of a localization of $A/J$, as the following analog of Lemma
\ref{shrinkEJ} shows.

\begin{lemma}  \label{EAtils} 
Impose Assumptions {\rm \ref{AHetc}}, and let $J \subset K$ in $\Hspec A$. Suppose there
exists a denominator set $\calEtil_{JK} \subseteq \calE_{JK}$ such that 
\begin{enumerate} 
\item $(L/J) \cap
\calEtil_{JK} \ne \varnothing$ for all $H$-primes $L \supseteq J$ such that $L \nsubseteq K$. 
\end{enumerate} 
Then
\begin{equation}  \label{Zunchanged} 
Z_{JK} = Z \bigl( (A/J)[\calEtil_{JK}^{-1} ] \bigr). 
\end{equation} 
\end{lemma}

\begin{proof} We may assume that $J=0$.

Consider an element $z \in Z (A[\calEtil_{JK}^{-1} ])$. Then $z \in Z(\Fract A)$ and $z= ac^{-1}$ for some $a \in A$
and $c \in \calEtil_{JK}$. Since then $c \in \calE_J$, we have $z \in A_J$ and hence $z \in Z(A_J)$. Moreover, $c \in
\calE_{JK}$ and $zc \in A$, whence $z \in Z_{JK}$.

Conversely, given $z \in Z_{JK}$, we have $z \in Z(A_J)$ and $zb \in A$ for some $b \in \calE_{JK}$. Choose primes
$L_1,\dots,L_n$ minimal over $AbA$ such that $L_1L_2 \cdots L_n \subseteq AbA$. Since $b$ is homogeneous, the $L_i$
are $H$-primes, and since $b \notin K$, no $L_i$ is contained in $K$. By hypothesis (1), there exist elements $c_i \in
L_i \cap \calEtil_{JK}$ for $i=1,\dots,n$. Now $c := c_1c_2 \cdots c_n \in \calEtil_{JK}$ and $c \in AbA$. Moreover,
$zc \in zAbA = AzbA \subseteq A$, so we can write $z = ac^{-1}$ with $a := zc \in A$. This shows that $z \in
A[\calEtil_{JK}^{-1} ]$. Since also $z \in Z(\Fract A)$, we conclude that $z \in Z (A[\calEtil_{JK}^{-1} ])$. This
establishes the last equality of \eqref{Zunchanged}. 
\end{proof}

\begin{lemma}  \label{ZJK+map} 
Impose Assumptions {\rm\ref{AHetc}}, let $J \subset K$ in $\Hspec A$, and let
$\pi_{JK}$ denote the quotient map $A/J \rightarrow A/K$.

There is a unique $k$-algebra map $f_{JK} : Z_{JK} \rightarrow Z(A_K)$ such that 
\begin{equation}  \label{fdef}
f_{JK}(z) = \pi_{JK}(zc)\pi_{JK}(c)^{-1} \text{\;for\;} z \in Z(A_J) \text{\;and\;} c\in \calE_{JK} \text{\;with\;} zc
\in A/J. 
\end{equation} 
\end{lemma}

\begin{proof} Assuming existence, uniqueness of $f_{JK}$ is clear.

There is no loss of generality in assuming that $J=0$. Write $\pi := \pi_{JK}$ and $f:= f_{JK}$. Set $\calE :=
\calE_{JK}$, and note that $\pi(c)$ is invertible in $A_K$ for all $c \in \calE$. We will also use the fact that, by Theorem \ref{strat.thm}(a), 
$\pi(\calE) = \calE_K$ is a denominator set in $A/K$.

We wish to define $f$ first as a map $Z_{JK} \rightarrow A_K$, via the rule \eqref{fdef}. Suppose that $z\in Z(A_J)$
and $c_1,c_2\in \calE$ such that $zc_1,zc_2 \in A$. Since $c_1,c_1z,zc_i \in A$, we see that 
$$\pi(c_1) \pi(zc_i)
\pi(c_i)^{-1} = \pi(c_1zc_i) \pi(c_i)^{-1} = \pi(c_1z)$$ 
for $i=1,2$, whence $\pi(zc_1) \pi(c_1)^{-1} = \pi(zc_2)
\pi(c_2)^{-1}$. Therefore we have a well defined map $f: Z_{JK} \rightarrow A_K$ defined by \eqref{fdef}.

Next, we show that $f$ maps $Z_{JK}$ to $Z(A_K)$. It suffices to show, for each $z\in Z_{JK}$,  that $f(z)$ commutes
with $\pi(a)$ for all $a\in A$, since $A_K = \pi(A)[\pi(\calE)^{-1}]$. Choose $c\in \calE$ such that $zc\in A$, and
observe that $\pi(zc)\pi(c) = \pi(c)\pi(zc)$, whence 
$$\pi(c)^{-1} \pi(zc) = \pi(zc) \pi(c)^{-1} = f(z).$$ 
Since also
$\pi(c) \pi(azc) = \pi(zca) \pi(c)$, we see that 
$$\pi(a) f(z) = \pi(azc) \pi(c)^{-1} = \pi(c)^{-1} \pi(zca) = f(z)
\pi(a).$$ 
Thus $f(z) \in Z(A_K)$, as desired.

Finally, let $z_1,z_2 \in Z_{JK}$, and choose $c_1,c_2 \in \calE$ such that $z_ic_i \in A$ for $i=1,2$. In view of
\eqref{ZJKsubalg} and the centrality of $f(z_2)$, we find that 
\begin{align*} 
f(z_1z_2) &= \pi(z_1z_2c_1c_2)
\pi(c_1c_2)^{-1} = \pi(z_1c_1) \pi(z_2c_2) \pi(c_2)^{-1} \pi(c_1)^{-1}  \\
 &= \pi(z_1c_1) f(z_2) \pi(c_1)^{-1} = \pi(z_1c_1) \pi(c_1)^{-1} f(z_2) = f(z_1) f(z_2)  \\
f(z_1+z_2) &= \pi((z_1+z_2)c_1c_2) \pi(c_1c_2)^{-1}  \\
 &= \pi(z_1c_1) \pi(c_2) \pi(c_2)^{-1} \pi(c_1)^{-1} + \pi(c_1) \pi(z_2c_2) \pi(c_2)^{-1} \pi(c_1)^{-1}  \\
  &= f(z_1) + \pi(c_1) f(z_2) \pi(c_1)^{-1} = f(z_1) + f(z_2).
\end{align*} 
Since it is clear from \eqref{fdef} that $f(1)=1$, we conclude that $f$ is indeed an algebra
homomorphism. 
\end{proof}

Given a homomorphism $d : R \rightarrow S$ between commutative $k$-algebras, where $S$ is affine but $R$ might not be,
we shall use the same notation $d^\circ$ for both of the comorphisms 
$$\max S \longrightarrow \max R
\qquad\text{and}\qquad \spec S \longrightarrow \spec R$$ 
corresponding to $d$.

\begin{conjecture}  \label{fJKgJKconj} 
Impose Assumptions {\rm\ref{AHetc}}, and let $J \subset K$ in $\Hspec A$.
Identify $\spec_J A$, $\spec_K A$, $\prim_J A$, $\prim_K A$ with $\spec Z(A_J)$, $\spec Z(A_K)$, $\max Z(A_J)$, $\max
Z(A_K)$ via the homeomorphisms of Theorem {\rm\ref{strat.thm}}.

Define the subalgebra $Z_{JK} \subseteq Z(A_J)$ as in Definition {\rm\ref{ZJK}} and the homomorphism $f_{JK} : Z_{JK}
\rightarrow Z(A_K)$ as in Lemma {\rm\ref{ZJK+map}}. Finally, define $g_{JK} : Z_{JK} \rightarrow Z(A_J)$ to be the
inclusion map. We conjecture that the maps $\phi^s_{JK}$ and $\phi^p_{JK}$ defined in {\rm\eqref{phiJKdef}} are both
given by the formula 
\begin{equation}  \label{JKconj} 
\phi^\bullet_{JK} = f_{JK}^\circ \hightop g_{JK}^\circ \,.
\end{equation} 
\end{conjecture}

In all the examples we have computed, the algebras $Z_{JK}$ are affine, so that the homomorphisms $f_{JK}$ and $g_{JK}$ arise from morphisms among the affine varieties $\max Z(A_J)$ and $\max Z_{JK}$. Thus, if Conjecture \ref{fJKgJKconj} and the aforementioned affineness hold, the topological spaces $\spec A$ and $\prim A$ are determined (via the framework of Section \ref{strat.topo}) by a finite amount of classical data.


\section{Reduction to inclusion control}
\label{control}

Here we establish conditions under which Conjecture \ref{fJKgJKconj} holds. These conditions, expressed in terms of inclusions involving certain prime ideals, are shown to hold when suitable prime ideals in factor algebras are generated by normal elements. As a first instance, we verify the latter conditions in the case of $\Oq(GL_2(k))$.

\begin{proposition}  \label{conjreduc} 
Impose Assumptions {\rm\ref{AHetc}}, and let $J \subset K$ in $\Hspec A$. Write
$Z_{JK} \cdot \calE_{JK} = \{ zc \mid z \in Z_{JK},\; c \in \calE_{JK} \}$.

{\rm(a)} Conjecture {\rm\ref{fJKgJKconj}} holds for $\phi^s_{JK}$ if and only if 
\begin{equation}  \label{inclwish}
(P/J) \cap Z_{JK} \cdot \calE_{JK} \subseteq Q/J \implies P \subseteq Q 
\end{equation} 
for all $P \in \spec_J A$ and
$Q \in \spec_K A$.

{\rm(b)} Conjecture {\rm\ref{fJKgJKconj}} holds for $\phi^p_{JK}$ if and only if the implication {\rm\eqref{inclwish}}
holds for all $P \in \spec_J A$ and $Q \in \prim_K A$.

{\rm(c)} Conjecture {\rm\ref{fJKgJKconj}} holds for $\phi^s_{JK}$ if and only if it holds  for $\phi^p_{JK}$.
\end{proposition}

\begin{proof} Since the closed sets in $\spec_L A$ and $\prim_L A$, for $H$-primes $L \supseteq J$, have the forms
\begin{align*} V(I)\cap \spec_L A &= V(I+J)\cap \spec_L A  \\ V_p(I)\cap \prim_L A &= V_p(I+J)\cap \prim_L A
\end{align*} 
for ideals $I$ of $A$, there is no loss of generality in assuming that $J=0$.

Let us label the homeomorphism $\spec_J A \rightarrow \spec Z(A_J)$ of Theorem \ref{strat.thm} in the form $T \mapsto
T^* := TA_J \cap Z(A_J)$, and similarly for the homeomorphism $\spec_K A \rightarrow \spec Z(A_K)$. The restrictions
of these maps to homeomorphisms from $\prim_J A$ and $\prim_K A$ onto $\max Z(A_J)$ and $\max Z(A_K)$, respectively,
are then also given in the form $T \mapsto T^*$.

(a) We are aiming to characterize the condition 
\begin{equation}  \label{expandJKconj} 
\phi^s_{JK}(Y) = (f_{JK}^\circ
\hightop g_{JK}^\circ)(Y) \text{\;for all\;} Y \in \CL(\spec_J A) 
\end{equation} 
by means of \eqref{inclwish}. Any $Y
\in \CL(\spec_J A)$ has the form $Y = V(I) \cap \spec_J A$ for some ideal $I$ of $A$. Now $V(I) = V(P_1) \cup
\cdots\cup V(P_n)$ where $P_1,\dots, P_n$ are the primes of $A$ minimal over $I$, so $Y$ is the union of the closed
sets $Y_i := V(P_i) \cap \spec_J A$. Since $\phi^s_{JK}$ and $f_{JK}^\circ \hightop g_{JK}^\circ$ preserve finite
unions, they agree on $Y$ if and only if they agree on each $Y_i$. Thus, \eqref{expandJKconj} holds if and only if
$\phi^s_{JK}(Y) = (f_{JK}^\circ \hightop g_{JK}^\circ)(Y)$ for all $Y = V(P)\cap \spec_J A$, where $P$ is a prime of
$A$ that contains $J$. If $P \notin \spec_J A$, then $P$ must lie in $\spec_L A$ for some $H$-prime $L \supsetneq J$,
in which case $Y$ is empty. That case is no problem, since $\phi^s_{JK}(\varnothing) = \varnothing = (f_{JK}^\circ
\hightop g_{JK}^\circ)(\varnothing)$. Hence, we conclude that \eqref{expandJKconj} holds if and only if
\begin{equation}  \label{VPreduce} 
\begin{aligned} 
\phi^s_{JK}(Y) &= (f_{JK}^\circ \hightop g_{JK}^\circ)(Y)
\text{\;for all\;} Y \text{\;of the form\;}  \\ Y &= V(P)\cap \spec_J A \text{\;with\;} P \in \spec_J A. 
\end{aligned}
\end{equation}

We next characterize the sets $\phi^s_{JK}(Y)$ and $(f_{JK}^\circ \hightop g_{JK}^\circ)(Y)$ appearing in
\eqref{VPreduce}, i.e., we assume that $Y = V(P)\cap \spec_J A$ for some $P \in \spec_J A$. Since $P \in Y \subseteq
V(P)$ and $\overline{ \{P\} } = V(P)$, we see that $\Ybar = V(P)$, and hence 
\begin{equation}  \label{phiJKY}
\phi^s_{JK}(Y) = V(P) \cap \spec_K A. 
\end{equation}

For $Q \in \spec_K A$, we have $Q \in (f_{JK}^\circ \hightop g_{JK}^\circ)(Y)$ if and only if $f_{JK}^\circ(Q^*) \in
\overline{ g_{JK}^\circ(Y) }$. On one hand, $f_{JK}^\circ(Q^*) = f_{JK}^{-1}(Q^*)$. On the other hand, since the set
$g_{JK}^\circ(Y) = \{ T^*\cap Z_{JK} \mid T \in Y \}$ has a unique smallest element, namely $P^*\cap Z_{JK}$, the
closure of $g_{JK}^\circ(Y)$ in $\spec Z_{JK}$ is just the set of primes of $Z_{JK}$ that contain $P^*\cap Z_{JK}$.
Thus, 
$$Q \in (f_{JK}^\circ \hightop g_{JK}^\circ)(Y) \iff f_{JK}^{-1}(Q^*) \supseteq P^*\cap Z_{JK}.$$ Note that
$P^*\cap Z_{JK} = PA_J \cap Z(A_J) \cap Z_{JK} = PA_J \cap Z_{JK}$. Since $f_{JK}(P^*\cap Z_{JK}) \subseteq Z(A_K)$,
we have $f_{JK}(P^*\cap Z_{JK}) \subseteq Q^*$ if and only if $f_{JK}(P^*\cap Z_{JK}) \subseteq QA_K$. Hence, 
$$Q \in
(f_{JK}^\circ \hightop g_{JK}^\circ)(Y) \iff f_{JK}(PA_J\cap Z_{JK}) \subseteq QA_K \,.$$

Given $Q \in \spec_K A$, we want to show that $f_{JK}(PA_J\cap Z_{JK}) \subseteq QA_K$ if and only if $P \cap Z_{JK}
\cdot \calE_{JK} \subseteq Q$. To do so, we first observe that 
\begin{equation}  \label{PAJcapZJK} 
PA_J\cap Z_{JK} =
\{ pc^{-1} \mid p \in P,\; c\in \calE_{JK} \} \cap Z_{JK} \,. 
\end{equation} 
The inclusion $(\supseteq)$ is clear. If
$z \in PA_J\cap Z_{JK}$, there is some $c \in \calE_{JK}$ such that $zc \in A$, whence $zc \in PA_J \cap A = P$. This
establishes $(\subseteq)$ and \eqref{PAJcapZJK}. Consequently, 
$$f_{JK}(PA_J\cap Z_{JK}) = \{ \pi(p) \pi(c)^{-1} \mid
p \in P,\; c\in \calE_{JK},\; pc^{-1} \in Z_{JK} \}.$$ 
For $p \in P$ and $c\in \calE_{JK}$, we have $\pi(p)
\pi(c)^{-1} \in QA_K$ if and only if $\pi(p) \in QA_K$, if and only if $\pi(p) \in QA_K \cap (A/K) = Q/K$, if and only
if $p \in Q$. Thus, 
\begin{equation*} 
\begin{aligned} 
f_{JK}&(PA_J\cap Z_{JK}) \subseteq QA_K  \\
 &\iff \{ p \in P\mid pc^{-1} \in Z_{JK} \text{\; for some\;} c \in \calE_{JK} \} \subseteq Q \\
 &\iff P \cap Z_{JK} \cdot \calE_{JK} \subseteq Q,
 \end{aligned}
 \end{equation*}
as desired.

On combining the results above, we obtain 
\begin{equation}  \label{descrhightopY} 
(f_{JK}^\circ \hightop
g_{JK}^\circ)(Y) = \{ Q \in \spec_K A \mid P \cap Z_{JK} \cdot \calE_{JK}  \subseteq Q\}. 
\end{equation} 
It is clear
from \eqref{phiJKY} and \eqref{descrhightopY} that $\phi^s_{JK}(Y) \subseteq (f_{JK}^\circ \hightop g_{JK}^\circ)(Y)$.
Therefore \eqref{expandJKconj} holds if and only if 
\begin{equation}  \label{inclcond} 
\{ Q \in \spec_K A \mid P \cap
Z_{JK} \cdot \calE_{JK} \subseteq Q\} \subseteq V(P) 
\end{equation} 
for all $P \in \spec_J A$ and $Q\in \spec_K A$.
This completes the proof of (a).

(b) The proof is the same as for (a), modulo changing $V(-)$ to $V_p(-)$ throughout, except for two points. Namely, if
$P \in \spec_J A$ and $Y = V_p(P) \cap \prim_J A$, we need to know that $\Ybar = V_p(P)$ in $\prim A$ and that
\begin{equation}  \label{gJKYclosure} 
\overline{g_{JK}^\circ(Y)} = \{ M \in \max Z_{JK} \mid M \supseteq P^*\cap
Z_{JK} \} 
\end{equation} 
in $\max Z_{JK}$. The first statement is given by \eqref{VpPprimclosure}.

Now set $Y^* := \{ T^* \mid T \in Y \}$, which is a closed subset of $\max Z(A_J)$. Obviously $T^* \supseteq P^*$ for
all $T^* \in Y^*$. On the other hand, if $M \in \max Z(A_J)$ with $M \supseteq P^*$, then $M=T^*$ for some $T \in
\prim_J A$, and $T^*$ belongs to the closure of $\{ P^* \}$ in $\spec Z(A_J)$. It follows that $T$ must belong to the
closure of $\{ P \}$ in $\spec_J A$, yielding $T \supseteq P$ and $T \in Y$. Thus, 
$$Y^* = \{ M \in \max Z(A_J) \mid M
\supseteq P^* \}.$$ 
Since $Z(A_J)$ is a commutative affine algebra, it is a Jacobson ring, and so we must have $P^* =
\bigcap Y^*$. Consequently, 
$$P^* \cap Z_{JK} = \bigcap\, \{T^* \cap Z_{JK} \mid T \in Y \} = \bigcap g_{JK}^\circ(Y),$$
and \eqref{gJKYclosure} follows.

(c) If \eqref{inclwish} holds for $P \in \spec_J A$ and $Q \in \spec_K A$, then it holds a priori for $P \in \spec_J
A$ and $Q \in \prim_K A$. Conversely, assume that \eqref{inclwish} holds for $P \in \spec_J A$ and $Q \in \prim_K A$.
Let $P \in \spec_J A$ and $Q \in \spec_K A$ such that $(P/J) \cap Z_{JK} \cdot \calE_{JK} \subseteq Q/J$. If $Q' \in
\prim_K A$ and $Q \subseteq Q'$, then $(P/J) \cap Z_{JK} \cdot \calE_{JK} \subseteq Q'/J$, and so $P \subseteq Q'$ by
our assumption. By \eqref{VpPprimclosure}, the intersection of those $Q' \in \prim_K A$ that contain $Q$ equals $Q$,
whence $P \subseteq Q$. This verifies that \eqref{inclwish} holds for $P \in \spec_J A$ and $Q \in \spec_K A$.
\end{proof}

\begin{proposition}  \label{normgen} 
Impose Assumptions {\rm\ref{AHetc}}, let $J \subset K$ in $\Hspec A$, and let $P
\in \spec_J A$. If $P/J$ is generated by some set of normal elements of $A/J$, then {\rm\eqref{inclwish}} holds for
all $Q \in \spec_K A$. 
\end{proposition}

\begin{proof} We may assume that $J = 0$.

Suppose $Q \in \spec_K A$ and $P \nsubseteq Q$. Then there is a normal element $p \in P \setminus Q$. Write $p = c_1+
\cdots+ c_n$ where the $c_i$ are nonzero homogeneous elements with distinct degrees. Since $p$ is not in $Q$, it is
not in $K$, so the $c_i$ cannot all lie in $K$. We may assume that $c_1 \notin K$. By standard results (e.g.,
\cite[Proposition 6.20]{Ymem}), all the $c_i$ are normal; in fact, there is an automorphism $\phi$ of $A$ such that
$pa = \phi(a)p$ and $c_ia = \phi(a)c_i$ for all $a \in A$ and all $i$. In particular, $c_1$ is regular in $A$ and
regular modulo $K$, so that $c_1 \in \calE _J \cap \calE_{JK}$.

For any $a\in A$, we have $pc_1^{-1}\phi(a) = pac_1^{-1} = \phi(a)pc_1^{-1}$ in $\Fract A$. Hence, the element $z :=
pc_1^{-1}$ lies in $Z(A_J)$. The fact that $zc_1= p \in A$ now implies $z \in Z_{JK}$. Consequently, $p \in Z_{JK}
\cdot \calE_{JK}$, and therefore $P\cap Z_{JK} \cdot \calE_{JK} \nsubseteq Q$. 
\end{proof}

\begin{example}  \label{oqgl2} 
Let $A = \Oq(GL_2(k))$ with $q \in \kx$ not a root of unity, and use the standard
abbreviations for the generators of $A$, namely 
$$\begin{matrix} a&b\\ c&d \end{matrix} \quad := \quad \begin{matrix}
X_{11}&X_{12}\\ X_{21}&X_{22} \end{matrix}$$ 
and $\Delta^{-1}$, where $\Delta := ad-qbc$ denotes the quantum
determinant in $A$. There is a standard rational action of $H = (\kx)^4$ on $A$ such that 
\begin{equation}
\label{kx4action} 
(\alpha_1,\alpha_2,\beta_1,\beta_2) . X_{ij} = \alpha_i \beta_j X_{ij} \qquad\text{for\;} i,j=1,2.
\end{equation} 
As is well known, $A$ has exactly four $H$-primes, and the poset $\Hspec A$ may be displayed in the
following form, where we abbreviate the descriptions of the $H$-prime ideals by omitting angle brackets and commas.
For instance, $bc$ stands for $\langle b,c\rangle$. 
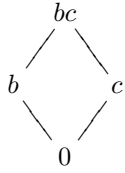
$$\xymatrixrowsep{1.2pc}\xymatrixcolsep{0.6pc} \xymatrix{
 &bc \edge[dl] \edge[dr]  \\
b \edge[dr] &&c \edge[dl]  \\
 &0
}$$ 
Finally, $A$ satisfies the noncommutative Nullstellensatz by \cite[Corollary II.7.18]{BG}, and so Assumptions
\ref{AHetc} hold.

Define the following multiplicative sets consisting of homogeneous normal elements: 
\begin{align*}
 \calEtil_0 &:=
\{\kx b^\bullet c^\bullet \Delta^\bullet \} \subseteq \calE_0  &\calEtil_b &:= \{ \kx c^\bullet \Delta^\bullet \}
\subseteq \calE_b  \\ \calEtil_c &:= \{ \kx b^\bullet \Delta^\bullet \} \subseteq \calE_c  &\calEtil_{bc} &:= \{ \kx
\Delta^\bullet \} \subseteq \calE_{bc} \,, 
\end{align*} 
where $x^\bullet$ abbreviates ``arbitrary nonnegative powers
of $x$'' and elements are interpreted as cosets where appropriate, and set $\Atil_J := (A/J)[\calEtil_J^{-1}]$.
Observe that each nonzero $H$-prime of $A/J$ has nonempty intersection with $\calEtil_J$. Hence,   Lemma
\ref{shrinkEJ}(c) shows that $Z(\Atil_J) = Z(A_J)$. These centers have the following forms: 
\begin{align*} 
Z(A_0) &=
k[(bc^{-1})^{\pm1}, \Delta^{\pm1}]  &Z(A_b) &= k[(ad)^{\pm1}]  \\ Z(A_c) &= k[(ad)^{\pm1}]  &Z(A_{bc}) &=
k[a^{\pm1},d^{\pm1}]. 
\end{align*}

Next, set $\calEtil_{J,K} := \calEtil_J \setminus K$ for $H$-primes $J \subset K$, and observe that \begin{align*}
\calEtil_{0,b} &= \{ \kx c^\bullet \Delta^\bullet \}  &\calEtil_{0,c} &= \{ \kx b^\bullet \Delta^\bullet \}
&\calEtil_{0,bc} &= \{ \kx \Delta^\bullet \}  \\ \calEtil_{b,bc} &= \{ \kx (ad)^\bullet \} &\calEtil_{c,bc} &= \{ \kx
(ad)^\bullet \} \,. 
\end{align*} 
Moreover, $\pi_{J,K}(\calEtil_{J,K}) = \calEtil_K$, and hence $Z_{J,K} = Z \bigl(
(A/J)[\calEtil_{J,K}^{-1}] \bigr)$ by Lemma \ref{EAtils}. These algebras have the following descriptions:
\begin{align*} 
Z_{0,b} &= k[ bc^{-1}, \Delta^{\pm1}]  &Z_{0,c} &= k[ b^{-1}c, \Delta^{\pm1}]  &Z_{0,bc} &= k[
\Delta^{\pm1}]  \\ Z_{b,bc} &= k[ (ad)^{\pm1}]  &Z_{c,bc} &= k[ (ad)^{\pm1}]  \,. 
\end{align*} 
The maximal ideal
spaces of the $Z(A_J)$ and the $Z_{J,K}$ are copies of the affine varieties $\kx$, $(\kx)^2$, and $k \times \kx$. We
can picture these spaces together with the associated maps $f_{J,K}^\circ$ and $g_{J,K}^\circ$ as in Figure
\ref{figprimoqgl2} below.

\begin{figure}[htp] 
$$\xymatrixrowsep{4pc}\xymatrixcolsep{1.8pc} \xymatrix{ \save [0,0]+(-5,5);[2,1]+(10,-5)
**\frm{.}\restore &\kx &&\save [0,0]+(-5,5);[2,2]+(5,-5) **\frm{.}\restore &(\kx)^2 \ar[lll]_(0.62){\txt{mult}}
\ar[dllll]_(0.72){\txt{mult}\;} \ar[rrr]^(0.62){\txt{mult}} &&&\kx \save [0,0]+(-10,5);[2,0]+(10,-5) **\frm{.}\restore
\\ \kx &&&\kx \ar[ull]^(0.35){\txt{id}} \ar[dll]_(0.29){(0,-)} &&\kx \ar[urr]^(0.45){\txt{id}} \ar[drr]^(0.29){(0,-)}
\\ \dropdownright{\max Z_{J,K}} \dropdownrightmore{(J=0,0,b)} {\save+<8ex,-14ex>\drop{(K=b,bc,bc)}\restore}
&k\times\kx  &&&(\kx)^2 \dropdown{\max Z(A_J)} \dropdownmore{(J=0,b,c,bc)}
 \ar[lll]_(0.62){\txt{incl}} \ar[ullll]_(0.77){\txt{pr}_2} \ar[rrr]^(0.62){\txt{incl}} &&&k\times\kx \dropdown{\max
 Z_{J,K}} \dropdownmore{(J=0,c)} {\save+<0ex,-14ex>\drop{(K = c,bc)}\restore}
 }$$
\caption{$\prim \Oq(GL_2(k))$ with spaces $\max Z_{J,K}$ and maps $f_{J,K}^\circ$, $g_{J,K}^\circ$}
\label{figprimoqgl2} 
\end{figure}
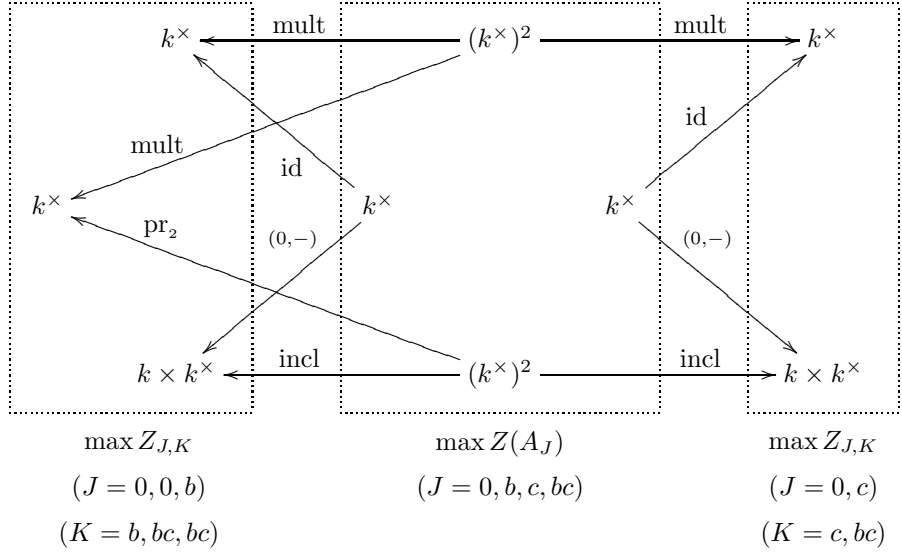

In order to see that the topology on $\prim A$ is determined by this picture, and similarly for the topology on $\spec
A$, we need to show that Conjecture \ref{fJKgJKconj} holds. This will follow from Propositions \ref{conjreduc} and
\ref{normgen} provided we verify that 
\begin{enumerate} \item[$(*)$] 
For each $J \in \Hspec A$ and each non-minimal $P
\in \spec_J A$, the ideal $P/J$ of $A/J$ is
    generated by normal elements.
\end{enumerate}

In the case $J=b$, we find that $P = \langle b,\, ad-\mu \rangle$ for some $\mu \in \kx$. Then $P/J$ is normally
generated because $ad-\mu$ is normal (in fact, central) in $A/J$. The case $J=c$ is exactly analogous. In the case $J=
bc$, the algebra $A/J$ is commutative, so all its ideals are centrally generated.

Finally, consider the case $J=0$. The maximal elements of $\spec_0 A$ are of the form $\langle b-\lambda c,\,
\Delta-\mu \rangle$ for $\lambda,\mu \in \kx$. These ideals are normally generated because $b- \lambda c$ is normal
and $\Delta - \mu$ is central. The remaining nonzero elements of $\spec_0 A$ are height $1$ primes of $A$. Each of
these is generated by a normal element because $A$ is a noncommutative UFD \cite[Corollary 3.8]{LLR}. This finishes
the verification of $(*)$, and we conclude that
 Conjecture \ref{fJKgJKconj} holds for this example.
\end{example}


\section{Quantum $SL_3$} \label{oqsl3}

The purpose of this section is to verify Conjecture \ref{fJKgJKconj} for $\OqSLth$ for generic $q$, thus showing that $\spec \OqSLth$ and $\prim \OqSLth$ can be entirely determined by classical (i.e., commutative) algebro-geometric data. Side benefits of our analysis provide new information about the structure of prime factor algebras, such as that all $H$-prime factors of $\OqSLth$ are noncommutative UFDs. Moreover, as we show in the following section, all prime factors of $\OqSLth$ are Auslander-Gorenstein and GK-Cohen-Macaulay, extending a result of \cite{primOqSL3} from primitive factors to prime factors.

Throughout the section, let $A = \OqSLth$, with $q \in \kx$ not a root of unity, and let $X_{ij}$, for $i,j = 1,2,3$,
denote the standard generators of $A$. Recall that all prime ideals of $A$ are completely prime (e.g., \cite[Corollary
II.6.10]{BG}). There is a natural rational action of the torus 
\begin{equation}  \label{OqSL3torus} 
H := \{
(\alpha_1,\alpha_2,\alpha_3,\beta_1,\beta_2,\beta_3) \mid \alpha_1\alpha_2\alpha_3\beta_1\beta_2\beta_3 = 1 \}
\end{equation} on $A$ such that \begin{equation}  \label{HSLaction}
(\alpha_1,\alpha_2,\alpha_3,\beta_1,\beta_2,\beta_3) \cdot X_{ij} = \alpha_i \beta_j X_{ij} 
\end{equation} 
for
$(\alpha_1,\alpha_2,\alpha_3,\beta_1,\beta_2,\beta_3) \in H$ and $i,j = 1,2,3$. As is well known (see, e.g.,
\cite{primOqSL3}), $A$ has exactly $36$ $H$-primes. Since $A$ satisfies the noncommutative Nullstellensatz
\cite[Corollary II.7.18]{BG}, Assumptions \ref{AHetc} hold.

\begin{subsec} \label{Hprimesindexed} 
As in \cite{primOqSL3}, we index the $H$-primes of $A$ in the form $Q_{w_+,w_-}$
for $(w_+,w_-)$ in $S_3 \times S_3$. Generating sets for these ideals are given in Figure \ref{Hprimegens}, taken from
\cite[Figure 1]{primOqSL3} -- see \cite[Subsection 2.1 and Corollary 2.6]{primOqSL3}. In this figure, bullets and
squares stand for $1\times 1$ and $2\times 2$ quantum minors, respectively, while circles are placeholders.

It is clear from Figure \ref{Hprimegens} that the height of any $H$-prime $Q_w$ is at least as large as the number of
generators $g$ given for $Q_w$ in the figure. On the other hand, these generators can be arranged in a polynormal
sequence, and so by the noncommutative Principal Ideal Theorem (e.g., \cite[Theorem 4.1.11]{McRob}), $\hgt(Q_w) \le
g$. Thus, the height of $Q_w$ exactly equals the number of generators for $Q_w$ given in Figure \ref{Hprimegens}.

The $H$-primes of $A$ are permuted by various symmetries of $A$. We summarize the three discussed in
\cite[\S1.4]{primOqSL3}. First, there is the \emph{transpose automorphism} $\tau$, which satisfies $\tau(X_{ij}) =
X_{ji}$ for $i,j = 1,2,3$; moreover, $\tau([I|J]) = [J|I]$ for all quantum minors $[I|J]$. Second, there is the
antipode $S$ of $A$, which is an anti-automorphism such that $S([I|J]) = (-q)^{\Sigma I - \Sigma J} [\Jtil|\Itil]$ for
all $[I|J]$, where $\Itil := \{1,2,3\} \setminus I$ and similarly for $\Jtil$. Finally, there is an anti-automorphism
$\rho$ of $A$ such that $\rho(X_{ij}) = X_{4-j,4-i}$ for all $i$, $j$; it satisfies $\rho([I|J]) = [w_0(J)|w_0(I)]$
for all $[I|J]$, where $w_0 = (321)$ is the longest element of $S_3$. 

\begin{figure}[htp]
 $$\xymatrixrowsep{0.1pc}\xymatrixcolsep{0.1pc}
\xymatrix{ 
 &&&&\ar@{=}[28,0] \\
 &&&\dropgen{-0.5}0{w_-} \\
 &&& &&&\dropvarup0{321} &&&&\dropvarup0{231} &&&&\dropvarup0{312} &&&&\dropvarup0{132} &&&&\dropvarup0{213}
 &&&&\dropvarup0{123} \\
 &\dropvarup0{w_+} \\
\ar@{=}[0,28] &&&&& &&&&& &&&&& &&&&& &&&&& &&&\dropup{} \\ 
 &&& &&\ccc &&\plc &\hrzvrt &&\ccb &&\cbb &&\ccb &&\cbb \\
 &&\dropvarup0{321} & &&\ccc &&\plc &\hrz &&\ccc &&\ccc
&&\ccb &&\ccb \\
 &&& &&\ccc &&\ccc &&\ccc &&\ccc &&\ccc &&\ccc \\   \\
 &&& &&\ccc &&\plc &\hrzvrt &&\ccb &&\cbb &&\ccb &&\cbb \\
 &&\dropvarup0{231} & &&\ccc &&\plc &\hrz &&\ccc &&\ccc
&&\ccb &&\ccb \\
 &&& &&\bcc &&\bcc &&\bcc &&\bcc &&\bcc &&\bcc \\   \\
 &&& &&\ccc &&\plc &\place \edge[dd] \edge[r] &\place \edge[d] &&\ccb &&\cbb &&\ccb &&\cbb \\
 &&\dropvarup0{312} & &&\hrzvrt &\plc &&\place \edge[rr] \edge[d] &&\place
 &&\hrzvrt &\plc &&\hrzvrt &\plc &&\hrzvrt &\plb &&\hrzvrt &\plb \\
 &&& &&\hrz &\plc &&\hrz &\plc &&\hrz &\plc &&\hrz &\plc
 &&\hrz &\plc &&\hrz &\plc \\   \\
 &&& &&\ccc &&\plc &\hrzvrt &&\ccb &&\cbb &&\ccb &&\cbb \\
 &&\dropvarup0{132} & &&\bcc &&\plb &\hrz &&\bcc &&\bcc
&&\bcb &&\bcb \\
 &&& &&\bcc &&\bcc &&\bcc &&\bcc &&\bcc &&\bcc \\   \\
 &&& &&\ccc &&\plc &\hrzvrt &&\ccb &&\cbb &&\ccb &&\cbb \\
 &&\dropvarup0{213} & &&\ccc &&\plc &\hrz &&\ccc &&\ccc
&&\ccb &&\ccb \\
 &&& &&\bbc &&\bbc &&\bbc &&\bbc &&\bbc &&\bbc \\   \\
 &&& &&\ccc &&\plc &\hrzvrt &&\ccb &&\cbb &&\ccb &&\cbb \\
 &&\dropvarup0{123} & &&\bcc &&\plb &\hrz &&\bcc &&\bcc
&&\bcb &&\bcb \\
 &&& &&\bbc &&\bbc &&\bbc &&\bbc &&\bbc &&\bbc \\
 &&&&&
}$$ \caption{Generators for $H$-prime ideals of $\OqSLth$}   \label{Hprimegens} 
\end{figure}
\end{subsec}

Recall that a \emph{noncommutative unique factorization domain} in the sense of  \cite[Definition, p.50]{Cha},
\cite[Definition, p.23]{ChJo} is a domain $R$ such that each nonzero prime ideal of $R$ contains a \emph{prime
element}, i.e., a nonzero normal element $p$ such that $R/Rp$ is a domain.

\begin{theorem}  \label{HUFDfactors} 
For any $H$-prime $J$ of $A$, the algebra $A/J$ is a noncommutative UFD.
\end{theorem}

\begin{proof} By arguments of Launois, Lenagan, and Rigal in \cite[Proposition 1.6, Theorem 3.6]{LLR}
(cf.~\cite[Theorem 2.3]{GYcgl}), it suffices to show that each nonzero $H$-prime of $A/J$ contains a prime
$H$-eigenvector, i.e., for all $H$-primes $Q_v \supset Q_w$ in $A$ with $\hgt(Q_v/Q_w) = 1$, the ideal $Q_v/Q_w$ is
generated by a normal $H$-eigenvector. In $25$ cases, namely when $w_- \ne 231$ and $w_+ \ne 312$, this is clear by
inspection from Figure \ref{Hprimegens}. Since 
\begin{align*} 
S(Q_{321,231}) &= Q_{321,312}  &S(Q_{312,321}) &=
Q_{231,321}  &S(Q_{312,231}) &= Q_{231,312} 
\end{align*} 
the cases $w= (321,231), (312,321), (312,231)$ follow
immediately from the earlier cases. Next, observe that $S(Q_{132,231})$ must be an $H$-prime of height $3$. Since
$$S(Q_{132,231}) = \langle X_{13}, [23|13],[23|12] \rangle \subseteq Q_{132,312}$$ 
and $\hgt(Q_{132,312}) = 3$, we
find that $S(Q_{132,231}) = Q_{132,312}$. Hence, the case $w= (132,231)$ follows from the earlier cases. The cases
$$w= (213,231),\, (123,231),\, (312,132),\, (312,213),\, (312,123)$$
 are handled similarly.

Only the cases $w= (231,231),(312,312)$ remain. Since $\tau$ interchanges $Q_{231,231}$ and $Q_{312,312}$, it suffices
to deal with one of these cases. We concentrate on $w= (231,231)$.

There are four indices $v$ such that $Q_v$ is an $H$-prime of height $3$ containing $Q_w$. In two of these cases,
namely when $v= (132,231)$ or $v= (213,231)$, it is clear that $Q_v/Q_w$ is generated by a normal $H$-eigenvector. The
remaining two cases are when $v= (231,132)$ or $(231,213)$. Since $\rho(Q_w) = Q_w$ and $\rho(Q_{231,132}) =
Q_{231,213}$, we need only consider the case $v= (231,132)$.

Note that $X_{12}$ is normal modulo $Q_{321,312}$. Applying $S$, we find that $[13|23]$ is normal modulo
$Q_{321,231}$, and hence normal modulo $Q_w$. Next, observe that $S$ sends the ideal $K := Q_w + \langle [13|23]
\rangle$ to $Q_{312,132}$, which is an $H$-prime of height $3$, so $K$ must be an $H$-prime of height $3$. However, $K
\subseteq Q_v$ and $Q_v$ is an $H$-prime of height $3$, so we conclude that $K = Q_v$. This implies that $Q_v/Q_w$ is
generated by the normal $H$-eigenvector $[13|23] + Q_w$, completing the proof. 
\end{proof}

Recall that a \emph{polynormal regular sequence} in a ring $R$ is a sequence of elements $u_1,\dots,u_n$ such that
each $u_i$ is regular and normal modulo $\langle u_1,\dots,u_{i-1} \rangle$. If the $u_i$ are all normal in $R$, we
refer to $u_1,\dots,u_n$ as a \emph{regular normal sequence}.

\begin{theorem}  \label{P/Jnormgen} 
For any $J \in \Hspec A$ and $P \in \spec_J A$, the ideal $P/J$ is generated by
normal elements. In fact, $P/J$ is generated by a regular normal sequence, and thus $P$  is generated by a polynormal
regular sequence. 
\end{theorem}

\begin{proof} The argument of \cite[\S2.4(4)]{primOqSL3} shows that $J$ has a polynormal regular sequence of
generators, and so we only need to show that $P/J$ has a regular normal sequence of generators.

There is nothing to prove in case $P/J = 0$. If $P/J$ has height $1$, then $P/J$ is generated by a normal element $u$
because $A/J$ is a noncommutative UFD (Theorem \ref{HUFDfactors}), and $u$ is regular because $A/J$ is a domain.
Assume now that $\hgt(P/J) \ge 2$.

Write $J = Q_w$, and let $Q^+_w$ denote the corresponding $H$-prime in $\OqGLth$. According to \cite[Corollary 5.4,
Theorem 5.5]{primOqSL3}, the elements listed in position $w$ of \cite[Figure 6]{primOqSL3} give regular normal
sequences in $\OqGLth/Q^+_w$ and the ideals they generate cover all quotients $P^+/Q^+_w$ where $P^+ \in \prim_w
\OqGLth$. Consequently, the elements listed in position $w$ of \cite[Figure 7]{primOqSL3} are normal in $A/Q_w$ and
the ideals they generate cover all quotients $P'/Q_w$ where $P' \in \prim_w A$. Note that in all but three cases, the
number of elements listed is at most two. In these three cases, the quotients $P'/Q_w$ can be generated by two of the
three elements listed, since 
\begin{gather*} \begin{aligned} 
{}[23|23] - \alpha^{-1} &= -\alpha^{-1} [23|23] (X_{11} -
\alpha)  &&\quad (w = (132,132))  \\ [12|12] - \alpha^{-1} &= -\alpha^{-1} [12|12] (X_{33} - \alpha)  &&\quad (w =
(213,213)) \end{aligned}  \\ \begin{aligned} X_{33} - \alpha^{-1} \beta^{-1} &= -\alpha^{-1} X_{33} (X_{11}-\alpha)
-\alpha^{-1} \beta^{-1} X_{11}X_{33} (X_{22}-\beta)  \\
 &\qquad (w= (123,123)).
 \end{aligned}
\end{gather*} 
Thus, in all cases, $P'/Q_w$ can be generated by two or fewer normal elements, and we conclude that
$\hgt(P'/Q_w) \le 2$.

Observe next that in the four cases 
$$w= (132,123), \, (213,123), \, (123,132), \, (123,213),$$ 
the quotients $P'/Q_w$
where $P' \in \prim_w A$ can be generated by single normal elements, so they have height $1$.

Since the primitive ideals in $\spec_w A$ coincide with the maximal elements of that stratum, our assumption
$\hgt(P/Q_w) \ge 2$ implies that $P \in \prim_w A$. There are only six cases where this can occur: 
$$w = (321,321), \,
(231,231), \, (312,312), \, (132,132), \, (213,213), (123,123).$$ 
In the first three of these cases, the first element
of the regular normal sequence in position $w$ of \cite[Figure 6]{primOqSL3} is $D_q - \alpha$, where $D_q$ is the
quantum determinant and $\alpha \in \kx$. Choosing $\alpha = 1$, we find that the remaining elements listed -- i.e.,
those in position $w$ of \cite[Figure 7]{primOqSL3} -- give regular normal sequences in $A/Q_w$ and the ideals they
generate cover all quotients $P'/Q_w$ where $P' \in \prim_w A$. Thus, $P/Q_w$ is generated by a regular normal
sequence in these cases. This likewise holds in the case $w = (123,123)$, since in that case, $A/Q_w$ is a commutative
Laurent polynomial ring.

The cases $w= (132,132),(213,213)$ remain. In both of these cases, $A/Q_w$ is isomorphic to the algebra $B :=
\Oq(GL_2(k))$, via an isomorphism that carries $P/Q_w$ to a maximal element of $\spec_0 B$. As noted in Example
\ref{oqgl2}, the maximal elements of $\spec_0 B$ have the form $\langle b - \lambda c, \, \Delta - \mu \rangle$ for
$\lambda, \mu \in \kx$. The quotients $B/ \langle \Delta - \mu \rangle$ are isomorphic to $\Oq(SL_2(k))$, so they are
domains. Consequently, $(\Delta - \mu, \, b - \lambda c)$ is a regular normal sequence in $B$. Therefore $P/Q_w$ is
generated by a regular normal sequence in the final two cases. 
\end{proof}

We now see that Conjecture \ref{fJKgJKconj} holds in the present situation:

\begin{theorem}  \label{conjverifOqSL3} 
Let $A = \OqSLth$, with $q \in \kx$ not a root of unity and $k=\kbar$, and let
the torus $H$ of \eqref{OqSL3torus} act rationally on $A$ as in \eqref{HSLaction}. Then both cases of Conjecture
{\rm\ref{fJKgJKconj}} hold.
 \end{theorem}

\begin{proof} Theorem \ref{P/Jnormgen} and Propositions \ref{normgen}, \ref{conjreduc}. 
\end{proof}


\section{Homological applications}
\label{Homological}

We establish the announced homological conditions for prime factor algebras of $\OqSLth$ here, and then show that these conditions do not hold for all prime factors of quantized coordinate rings of larger algebraic groups.  We begin with the following consequence of Theorem \ref{P/Jnormgen}. It was obtained for primitive
 factor algebras in \cite[Theorem 6.1]{primOqSL3}.

\begin{theorem}  \label{homolconseq} 
Let $A = \OqSLth$, with $q \in \kx$ not a root of unity and $k=\kbar$. Then all
prime factor algebras of $A$ are Auslander-Gorenstein and GK-Cohen-Macaulay. 
\end{theorem}

\begin{proof} By Theorem \ref{P/Jnormgen}, any prime ideal $P$ of $A$ has a polynormal regular sequence of generators.
Moreover, $A$ is Auslander-regular and GK-Cohen-Macaulay (e.g., \cite[Proposition I.9.12]{BG}). It thus follows from
standard results, collected in \cite[Theorem 7.2]{primOqSL3}, that $A/P$ must be Auslander-Gorenstein and
GK-Cohen-Macaulay. 
\end{proof}

We now show that Theorem \ref{homolconseq} does not extend to $\mathcal{O}_q(G)$ for an arbitrary group $G$, but
rather is a consequence of the special circumstance that all the $H$-strata of $\OqSLth$ have dimension at most 2. We
also prove that Theorem \ref{homolconseq} cannot be improved so as to conclude that the prime factors of
$\OqSLth$ have finite global dimension. For these results we need the following lemma.

\begin{lemma}  \label{freeness} 
Impose Assumptions {\rm\ref{AHetc}}. For any $J \in \Hspec A$, the algebra $A_J$ is a
free module over its center. Moreover, there is a $Z(A_J)$-basis for $A_J$ that contains $1$. 
\end{lemma}

\begin{proof} Theorem \ref{strat.thm}(a) says that $A_J$ is $H$-simple, and thus also graded-simple with respect to
the $X(H)$-grading. The proof of \cite[Lemma II.3.7]{BG} shows that $Z(A_J)$ is a homogeneous subring of $A_J$, the
set 
$$\Gamma := \{ \chi \in X(H) \mid Z(A_J)_\chi \ne 0 \}$$ 
is a subgroup of $X(H)$, and the homogeneous subring $S
:= \bigoplus_{\chi \in \Gamma} \, (A_J)_\chi$ of $A_J$ is a free $Z(A_J)$-module with a basis containing $1$.

The graded-simplicity of $A_J$ implies that its identity component is simple, from which it follows that $A_J$ is
strongly graded. Choose a transversal $T$ for $\Gamma$ in $X(H)$ such that $1 \in T$, and observe that $A_J$ is a free
left $S$-module with basis $T$. Both conclusions of the lemma now follow. 
\end{proof}

\begin{subsec}  \label{OqGfacts}
Let $A= \Oq(G)$, with $q \in \kx$ not a root of unity and $k= \kbar$, where $G$ is $SL_n(k)$, $GL_n(k)$, or a connected, simply connected, semisimple complex algebraic group. There are standard choices for a $k$-torus $H$ acting rationally on $A$ by $k$-algebra automorphisms, as in \cite[\S\S II.1.15, II.1.16, II.1.18, Exer.~II.2.G]{BG}. The remaining parts of Assumptions \ref{AHetc} hold by \cite[Theorems I.2.10, I.8.18, II.5.14, II.5.17, Corollaries I.2.8,  II.4.12, II.7.18, II.7.20]{BG}.

There are $H$-strata of $\prim A$ with dimension $\rank G$, as follows. In case $G$ is $SL_n(k)$ or $GL_n(k)$, we can just let $J$ be the $H$-prime $\langle X_{ij} \mid i \ne j \rangle$ and observe that $A/J$ is a Laurent polynomial ring over $k$ in $n-1$ (respectively, $n$) variables. In this special case, $A/J = A_J = Z(A_J)$, and the stratum $\prim_J A$ has dimension $n-1$ (respectively, $n$), in view of Theorem \ref{strat.thm}. There are other strata with the same dimension, obtained for the $SL_n$ case as in the following paragraph, and then for the $GL_n$ case using the isomorphism $\Oq(GL_n(k)) \cong \OqSLn[z^{\pm1}]$ (e.g., \cite[Lemma II.5.15]{BG}).

In the remaining cases, choose $J = K_{w_+,w_-}$ in the notation of \cite[Proposition II.4.11]{BG}, with $w_+ = w_-$. Then \cite[Corollary II.4.15]{BG} shows that $Z(A_J)$ is a Laurent polynomial ring in $\rank G$ variables, so that, again, $\prim_J A$ has dimension $\rank G$ by Theorem \ref{strat.thm}. (In the case $w_+ = w_- = \ident$, we have $A/J = A_J = Z(A_J)$ as above.)
\end{subsec}

\begin{theorem} \label{infdim}
Let $A= \mathcal{O}_q(G)$, with $q \in \kx$ not a root of unity and $k = \kbar$, where $G$ is either a nontrivial connected, simply connected, semisimple complex algebraic group or $GL_n(k)$ for some $n \ge 2$. 

{\rm(a)} If $G$ is not $SL_2(k)$, then $A$ has a prime factor of infinite global dimension.

{\rm(b)} If $G$ is not $SL_2(k)$, $GL_2(k)$ or $SL_3(k)$, then $A$ has a prime factor of infinite injective dimension.  
\end{theorem}

\begin{proof} Let $H$ be the $k$-torus acting rationally on $A$ as in \S\ref{OqGfacts}.

(a) The hypothesis on $G$ guarantees that $\prim A$ contains an $H$-stratum  of dimension $t \geq 2$, by \S\ref{OqGfacts}; choose such a stratum, $\prim_J A$. Thus, $Z(A_J)$ is a Laurent polynomial algebra over $k$ in $t$ variables. We can therefore find a prime ideal $\mathfrak{p}$ of $Z(A_J)$ such that $Z(A_J)/\mathfrak{p}$ has infinite global dimension. (For example, we might take $\mathfrak{p} = \langle (x-1)^2 - (y-1)^3 \rangle$, where $x^{\pm 1}$, $y^{\pm 1}$ are the first two Laurent variables of $Z(A_J)$.) Now set $P = \mathfrak{p}A_J$, a prime ideal of $A_J$ by Theorem \ref{strat.thm}. We claim that
\begin{equation}\label{big} 
\mathrm{gl.dim.}(A_J/P) = \infty. 
\end{equation}

For, suppose to the contrary that $\mathrm{gl.dim.}(A_J/P) = d < \infty$. Let $M$ be any left $Z(A_J)/\mathfrak{p}$-module, and consider the $A_J/P$-module $A_J/P \otimes_{Z(A_J)/\mathfrak{p}} M$. By our supposition, this module has a finite resolution by $A_J/P$-projectives. But now Lemma \ref{freeness} ensures, first, that the terms of the resolution are $Z(A_J)/\mathfrak{p}$-projective, and second, that $M$ is a direct summand of $A_J/P \otimes_{Z(A_J)/\mathfrak{p}} M$ as $Z(A_J)/\mathfrak{p}$-modules. It follows that $M$ has projective dimension at most $d$; since $M$ was arbitrary, we conclude that $\mathrm{gl.dim.}(Z(A_J)/\mathfrak{p})$ is finite, a contradiction. Thus, \eqref{big} is proved.

Now let $Q$ be the prime ideal in $\spec_J A$ such that $(Q/J) A_J = P$. By Theorem \ref{strat.thm}, $P \cap (A/J) = Q/J$, so $A_J/P$ is an Ore localisation of $A/Q$, and hence \eqref{big} implies that $A/Q$ has infinite global dimension.

(b) Let $A = \Oq(G)$, where $G$ is as stated. Then, by \S\ref{OqGfacts}, $\prim A$ has at least one $H$-stratum $\prim_J A$ of dimension $t \geq 3$. That is, $Z(A_J)$ is a Laurent polynomial $k$-algebra in variables $x_1^{\pm 1}, \ldots, x_t^{\pm 1}$. Choose a prime ideal $\mathfrak{p}$ of $Z(A_J)$ such that $Z(A_J)/\mathfrak{p}$ is not Gorenstein. For example, letting $x^{\pm 1}$, $y^{\pm 1}$, $z^{\pm 1}$ be the first three generators of $Z(A_J)$, one can take $\mathfrak{p}$ to be the prime ideal 
$$\langle (x - 1)^4 - (y-1)^3, \, (y-1)^5 - (z-1)^4, \, (x-1)^5 - (z-1)^3 \rangle$$
 of $Z(A_J)$, by, e.g., \cite[Theorem 4.3.10]{BH}. The argument now proceeds in a manner similar to (a). In brief, let $P = \mathfrak{p}A_J$, a prime ideal of $A_J$. Suppose that $A_J/P$ has finite injective dimension as a left $A_J/P$-module, with resolution
\begin{equation}\label{inj} 
0 \longrightarrow A_J/P \longrightarrow E_0 \longrightarrow \cdots \longrightarrow E_m \longrightarrow 0. 
\end{equation}
In view of Lemma \ref{freeness}, a standard and easy argument shows that each $E_i$ is an injective $Z(A_J)/\mathfrak{p}$-module. Hence, $A_J/P$ and its direct summand $Z(A_J)/\mathfrak{p}$ have finite injective dimension as $Z(A_J)/\mathfrak{p}$-modules, a contradiction. Now let $Q$ be the prime ideal in $\spec_J A$ which corresponds to $P$. If $\mathrm{inj.dim.}(A/Q)$ were finite, then the same would be true of its localisation $A_J/P$, by the exactness of Ore localisation, and by the preservation of injectivity when localising at a set of normal elements in a noetherian ring \cite[Theorem 1.3]{GJ}. However we have just shown that this is not the case. Therefore $\mathrm{inj.dim.}(A/Q) = \infty$, as required.
\end{proof}


\section{$2\times2$ quantum matrices}
\label{oqm2}

In this final section, we verify Conjecture \ref{fJKgJKconj} for $\Oq(M_2(k))$ for generic $q$. There are side benefits almost the same as those obtained for $\OqSLth$: All prime factor algebras of $\Oq(M_2(k))$ are Auslander-Gorenstein and GK-Cohen-Macaulay, and all but one of the $H$-prime factors of $\Oq(M_2(k))$ are noncommutative UFDs. The exception, namely the quotient of $\Oq(M_2(k))$ modulo its quantum determinant, exhibits a phenomenon that has not been seen before to our knowledge: This domain is nearly a noncommutative UFD in that all but four of its height $1$ prime ideals are principal, while four are not.

Let $A= \Oq(M_2(k))$ throughout this section, with $q\in\kx$ a non-root of unity. Just as in Example \ref{oqgl2}, use
the standard abbreviations $a$, $b$, $c$, $d$ for the generators of $A$, let $\Delta$ denote the quantum determinant
in $A$, and let $H = (\kx)^4$ act rationally on $A$ as in \eqref{kx4action}. It is well known that $A$ has exactly
$14$ $H$-primes (e.g., \cite[\S3.6]{GLmurcia}). Since $A$ satisfies the noncommutative Nullstellensatz (e.g.,
\cite[Corollary II.7.18]{BG}), Assumptions \ref{AHetc} hold.

We display the poset $\Hspec A$ in Figure \ref{HprimesOqM2} below, where we again abbreviate descriptions of ideals by
omitting angle brackets and commas. Whenever we display quantities indexed by $\Hspec A$, we place the quantity
indexed by a given $H$-prime $J$ in the same relative position that $J$ occupies in Figure \ref{HprimesOqM2}. See
\eqref{EtilJ}, \eqref{ZOqM2J}, \eqref{maxgen}.

There is a transpose automorphism $\tau$ on $A$, which sends $a$, $b$, $c$, $d$ to $a$, $c$, $b$, $d$, and an
anti-automorphism $\rho$ which sends $a$, $b$, $c$, $d$ to $d$, $b$, $c$, $a$.

\begin{figure}[htp] 
$$\xymatrixrowsep{2.6pc}\xymatrixcolsep{2.5pc} \xymatrix{
 &&abcd\\
abd \edge[urr] &abc \edge[ur] &&bcd \edge[ul] &acd \edge[ull]\\ ab \edge[u] \edge[ur]
 &bd \edge[ul] \edge[urr]
 &bc \edge[ul] \edge[ur]
 &ac \edge[ull] \edge[ur]
 &cd \edge[ul] \edge[u]\\
 &b \edge[ul] \edge[u] \edge[ur]
 &\Delta \edge[ull] \edge[ul] \edge[ur] \edge[urr]
 &c \edge[ul] \edge[u] \edge[ur]\\
 &&0 \edge[ul] \edge[u] \edge[ur]
}$$
 \caption{$\Hspec \Oq(M_2(k))$}    \label{HprimesOqM2}
 \end{figure}

\begin{lemma} \label{partHufdOqM2} 
{\rm (a)} Let $J\subset K$ be $H$-primes of $A$ such that $\hgt(K/J)=1$. If $J\ne
\Delta$, then $K/J$ is generated by a normal element, while if $J= \Delta$, then $K/J$ cannot be generated by a normal
element.

{\rm (b)} $A/J$ is a UFD for all $H$-primes $J \ne \Delta$.

{\rm (c)} Every $H$-prime of $A$ can be generated by a polynormal regular sequence. 
\end{lemma}

\begin{proof} (a) The first statement is clear by inspection of Figure \ref{HprimesOqM2}. Now let $J= \Delta$ and $K=
ab$, and suppose that $K/J$ is generated by a normal element $u+J$. Then $K = \langle \Delta,u\rangle$. Since $(b,a)$
and $(\Delta,u)$ are polynormal sequences, the left ideals they generate are the same as the two-sided ideals. Hence,
there exist $r_1,r_2,s_1,s_2,t_1,t_2 \in A$ such that 
\begin{align*} a &= r_1\Delta+r_2u  &b &= s_1\Delta+s_2u  &u &=
t_1a+t_2b \;. 
\end{align*}

Transfer these equations to $A/cd$, which is a skew polynomial ring $k[a][b;\sigma]$. Here, $a= r_2u$ and $b= s_2u$,
from which it follows that $u$ is a nonzero scalar. Returning to $A$, we have $u= \alpha+ p_1c+ p_2d$ for some
$\alpha\in\kx$ and  $p_1,p_2\in A$. Thus, 
$$t_1a+ t_2b- p_1c- p_2d= \alpha.$$ 
This is impossible, since $A$ is a
positively graded ring in which $a$, $b$, $c$, $d$ are homogeneous of degree $1$.

Therefore $ab/ \Delta$ cannot be generated by a normal element. The cases $K= bd$, $ac$, $cd$ follow by symmetry (via
$\tau$ and $\rho$).

(b) This follows from part (a) and the arguments of \cite{LLR} (cf.~\cite[Theorem 2.3]{GYcgl}).

(c) This is clear from Figure \ref{HprimesOqM2}. 
\end{proof}

Define multiplicative sets $\calEtil_J \subseteq \calE_J$ for $J\in \Hspec A$ as in \eqref{EtilJ}. It follows from
Lemma \ref{shrinkEJ}(c) that $Z(A_J) = Z((A/J)[\calEtil_J^{-1}]$ for all $J$.

\begin{equation}  \label{EtilJ} {\begin{matrix}
 &&\{\kx\}  \\  \\
\{\kx c^\bullet\} &\{\kx d^\bullet\} &&\{\kx a^\bullet\} &\{\kx b^\bullet\}  \\  \\ \{\kx c^\bullet d^\bullet\} &\{\kx
a^\bullet c^\bullet\} &\{\kx a^\bullet d^\bullet\} &\{\kx b^\bullet d^\bullet\} &\{\kx a^\bullet b^\bullet\}  \\  \\
 &\{\kx a^\bullet c^\bullet d^\bullet\} &\{\kx a^\bullet b^\bullet c^\bullet d^\bullet\} &\{\kx a^\bullet b^\bullet
 d^\bullet\}  \\  \\
 &&\{\kx b^\bullet c^\bullet \Delta^\bullet\}
\end{matrix}} 
\end{equation}

Consider the following subalgebras of the algebras $A_J$ for $J \in \Hspec A$: 
\begin{equation}  \label{ZOqM2J}
{\begin{matrix}
 &&k  \\  \\
k[c^{\pm1}] &k[d^{\pm1}] &&k[a^{\pm1}] &k[b^{\pm1}]  \\  \\ k &k &k[a^{\pm1},d^{\pm1}] &k &k  \\  \\
 &k[(ad)^{\pm1}] &k[(bc^{-1})^{\pm1}] &k[(ad)^{\pm1}]  \\  \\
 &&k[(bc^{-1})^{\pm1},\Delta^{\pm1}]
\end{matrix}} 
\end{equation}

\begin{lemma} \label{ZOqM2verif} 
For each $J \in \Hspec A$, the algebra shown in position $J$ of \eqref{ZOqM2J} equals
the center of $A_J$. 
\end{lemma}

\begin{proof} We use the relations $Z(A_J) = Z((A/J)[\calEtil_J^{-1}]$ without comment.

The conclusion is clear if $J= abcd$, in which case $A/J = k$, and if $J$ is one of $abd$, $abc$, $bcd$, $acd$, in
which cases $A/J= k[c],k[d],k[a],k[b]$, respectively.

If $J$ is one of $ab$, $bd$, $ac$, $cd$, then $A/J$ is a copy of $\Oq(k^2)$. Since $\Fract \Oq(k^2)= \Oq((\kx)^2)$ has
center $k$, it follows that $Z(A_J) = k$ in these cases. The case $J= bc$ is clear, because then $A/J= k[a,d]$.

Now let $J= b$. In this case, $A_J$ is a quantum torus generated by $a^{\pm1}$, $c^{\pm1}$, $d^{\pm1}$, and we check
that monomials $a^ic^jd^l$ are central if and only if $j=0$ and $i=l$. Thus, $Z(A_J) = k[(ad)^{\pm1}]$. The same holds
when $J= c$, by symmetry.

Next, let $J= \Delta$. In $A_J$, we have $d= qa^{-1}bc$, and consequently $A_J$ is a quantum torus generated by
$a^{\pm1}$, $b^{\pm1}$, $c^{\pm1}$. We check that monomials $a^ib^jc^l$ are central if and only if $i= j+l= 0$. Thus,
$Z(A_J)= k[(bc^{-1})^{\pm1}]$.

Finally, let $J= 0$, and observe that $A[\calEtil_0^{-1}]$ is a quantum torus of rank $4$, with generators $a^{\pm1}$,
$b^{\pm1}$, $c^{\pm1}$, $\Delta^{\pm1}$. We check that monomials $a^ib^jc^l\Delta^m$ are central if and only if $i=
j+l= 0$. Thus, $Z(A_J)= k[(bc^{-1})^{\pm1},\Delta^{\pm1}]$. 
\end{proof}

Generating sets for the maximal ideals of the algebras $Z(A_J)$ can be given as follows, where $\alpha$, $\beta$,
$\delta$, $\gamma$, $\lambda$, $\mu$ are arbitrary nonzero scalars from $k$. 
\begin{equation}  \label{maxgen}
\begin{matrix}
 &&0  \\  \\
c-\gamma &d-\delta &&a-\alpha &b-\beta  \\  \\ 0 &0 &a-\alpha,\, d-\delta &0 &0  \\  \\
 &ad-\mu &b-\lambda c &ad-\mu  \\  \\
 &&b-\lambda c,\, \Delta-\mu
\end{matrix} 
\end{equation}

\begin{lemma}  \label{primJmodJnormgen} 
For each $J \in \Hspec A$, the elements listed in position $J$ of
\eqref{maxgen} form a regular normal sequence in $A/J$, and they generate a primitive ideal of $A/J$. These ideals
cover all quotients $P/J$ for $P \in \prim_J A$. 
\end{lemma}

\begin{proof} The statement about regular normal sequences is clear for $J \ne bc,0$. We deal with the cases $J =
bc,0$ later.

In view of Lemma \ref{ZOqM2verif} and Theorem \ref{strat.thm}, the quotients $P/J$ for $P \in \prim_J A$ are exactly
the ideals $QA_J \cap (A/J)$ where $Q$ is the ideal of $A/J$ generated by the elements in position $J$ of
\eqref{maxgen}, for some choice of scalars. Thus, we need to show that each such $Q$ equals $QA_J \cap (A/J)$. That
equality holds if $(A/J)/Q$ is $\calE_J$-torsionfree, so it will suffice to show that $Q$ is a prime ideal of $A/J$.
This is trivial when $J$ is one of $abcd$, $ab$, $bd$, $ac$, $cd$. The cases when $J$ is one of $abd$, $abc$, $bcd$,
$acd$, $bc$ are clear since then $A/J$ is a commutative polynomial ring, namely $k[c]$, $k[d]$, $k[a]$, $k[b]$,
$k[a,d]$, respectively.

The remaining four cases are based on the following claims: 
\begin{enumerate} 
\item $\langle b-\lambda c\rangle$ is a
prime ideal of $A$, for all $\lambda \in k$. 
\item $\langle b-\lambda c, \, \Delta -\mu \rangle$ is a prime ideal of
$A$ for all $(\lambda,\mu) \in k^2
    \setminus \{(0,0)\}$.
\end{enumerate} 
The case $J=b$ follows from (2) with $\lambda=0$ and $\mu\ne 0$, the case $J=c$ is symmetric to the
previous one, the case $J= \Delta$ follows from (2) with $\lambda \ne 0$ and $\mu=0$, and the case $J=0$ follows from
(2) with $\lambda,\mu \ne 0$. Moreover, it follows from (1) that $(b-\lambda c, \, \Delta- \mu)$ is a regular normal
sequence in $A$. Since $A/bc= k[a,d]$, we see that $(a-\alpha, \, d-\delta)$ is a regular normal sequence in $A/bc$.
Thus, what is left is to establish (1) and (2).

The algebra $A/\langle b-\lambda c\rangle$ has a presentation with generators $a$, $c$, $d$ and relations
\begin{align*} 
ac &= qca  &cd &= qdc  &ad-da = \lambda (q-q^{-1}) c^2 \,. 
\end{align*} 
It follows that this algebra is
an iterated skew polynomial ring of the form 
$$k[a] [c;\sigma_2] [d; \sigma_3,\delta_3],$$ 
and hence a domain. This
proves (1).

Now set $B := A/ \langle b-\lambda c, \, \Delta -\mu \rangle$, where $(\lambda,\mu) \in k^2 \setminus \{(0,0)\}$. This
algebra has a presentation with generators $a$, $c$, $d$ and relations 
\begin{align*} 
ac &= qca  &cd &= qdc  \\ ad &=
\lambda qc^2 + \mu  &da &= \lambda q^{-1}c^2 + \mu \,. 
\end{align*} 
It can also be viewed as generated by a copy of
the polynomial ring $k[c]$ together with elements $a$ and $d$ such that 
\begin{align*} 
dr &= \phi(r)d \quad \forall\,
r\in k[c]  &ar &= \phi^{-1}(r)a \quad \forall\, r\in k[c]  \\ ad &= \lambda qc^2 + \mu  &da &= \phi( \lambda qc^2 +
\mu ) \,, 
\end{align*} 
where $\phi$ is the $k$-algebra automorphism of $k[c]$ such that $\phi(c) = q^{-1}c$. Hence,
$B$ is a generalized Weyl algebra, of the form $k[c](\phi,\lambda qc^2 + \mu)$. Since $k[c]$ is a domain and $\lambda
qc^2 + \mu$ is nonzero, $B$ is a domain \cite[Proposition 1.3(2)]{Bav}. Therefore (2) holds. 
\end{proof}

\begin{theorem} \label{secondP/Jnormgen} 
Let $J \in \Hspec A$ and $P \in \spec_J A$. Then $P/J$ is generated by a
regular normal sequence, and $P$ is generated by a polynormal regular sequence. 
\end{theorem}

\begin{proof} Only the first statement needs to be proved, since $J$ is generated by a polynormal regular sequence
(Lemma \ref{partHufdOqM2}(c)). To prove the first statement, we may obviously assume that $P \ne J$.

First, assume that $J\ne bc,0$. In these cases, it follows from Lemma \ref{primJmodJnormgen} that $\hgt(P'/J) \le 1$
for all $P' \in \prim_J A$, and thus also for all $P' \in \spec_J A$ (since every element of $\spec_J A$ is contained
in an element of $\prim_J A$). The assumption $P \ne J$ then implies $P \in \prim_J A$, whence the lemma shows that
$P/J$ is generated by a normal element.

Now suppose that either $J=bc$ or $J=0$. In these cases, $A/J$ is a noncommutative UFD by Lemma \ref{partHufdOqM2}(b),
so if $P/J$ has height $1$, it must be generated by a normal element. From Lemma \ref{primJmodJnormgen}, we see that
$\hgt(P'/J) \le 2$ for all $P' \in \spec_J A$. Hence, if $\hgt(P/J) = 2$, then $P \in \prim_J A$, and the lemma
implies that $P/J$ is generated by a regular normal sequence. 
\end{proof}

Theorem \ref{secondP/Jnormgen} yields the same conclusions for $\Oq(M_2(k))$ that we obtained for $\OqSLth$ in Sections
\ref{oqsl3} and \ref{Homological}.

\begin{theorem}  \label{OqM2satconj} 
Let $A = \Oq(M_2(k))$, with  $q \in \kx$ not a root of unity and $k = \kbar$, and
let $H = (\kx)^4$ act rationally on $A$ in the standard fashion. Then both cases of Conjecture {\rm\ref{fJKgJKconj}}
hold. 
\end{theorem}

\begin{theorem}  \label{AGCMfactorsOqM2} 
Let $A = \Oq(M_2(k))$, with  $q \in \kx$ not a root of unity and $k = \kbar$.
Then all prime factor algebras of $A$ are Auslander-Gorenstein and GK-Cohen-Macaulay. 
\end{theorem}

\begin{remark}  \label{ncphenom} 
The results above show that the algebra $A/\Delta$ is very nearly a noncommutative
UFD. First, as noted in the proof of Theorem \ref{secondP/Jnormgen}, it follows from Lemma \ref{primJmodJnormgen} that
for any $P \in \spec_\Delta A$ with $\hgt(P/\Delta) = 1$, the prime $P/\Delta$ is generated by a normal element. These
are the primes $\langle \Delta, \, b-\lambda c\rangle /\Delta$, for $\lambda \in \kx$. The only other height $1$
primes in $A/\Delta$ are the $H$-primes $ab/\Delta$, $bd/\Delta$, $ac/\Delta$, and $cd/\Delta$, and by Lemma
\ref{partHufdOqM2}(a), none of these is generated by a normal element.

Thus, $A/\Delta$ has infinitely many height $1$ primes, all but four of which are principal. This is a noncommutative
phenomenon, in view of a theorem of Bouvier \cite{Bouv} which states that in a (commutative) Krull domain, the set of
non-principal height $1$ primes is either empty or infinite. To see that $A/\Delta$ is an appropriate noncommutative
analog of a Krull domain, recall that normal (i.e., integrally closed) commutative notherian domains are Krull
domains, and that  the standard analog of normality for a noncommutative noetherian domain is the property of being a
maximal order in its division ring of fractions. That $A/\Delta$ is a maximal order is one case of a theorem of Rigal
\cite[Th\'eor\`eme 2.2.7]{Rig}. 
\end{remark}


\end{document}